\documentclass[11pt]{article}

\usepackage[utf8]{inputenc}

\usepackage{graphicx}
\usepackage[ruled,vlined]{algorithm2e}
\usepackage{algorithmic}
\usepackage{epsfig}
\usepackage{epstopdf}
\usepackage{amssymb,amsmath}
\usepackage{amsmath}
\usepackage{amsthm}
\usepackage{amsfonts}
\usepackage{psfrag}
\usepackage{rotating}
\usepackage{latexsym}
\usepackage{dsfont}
\usepackage{stmaryrd}
\usepackage{amsthm}
\usepackage{amssymb}
\usepackage{amsmath}
\usepackage{mathtools}
\usepackage{mathrsfs}   
\usepackage{siunitx}
\usepackage{pgfpages}
\usepackage[Euler]{upgreek} 
\usepackage{isomath}
\usepackage{physics}
\usepackage{pifont} 
\usepackage{bbm}
\usepackage{empheq}
\usepackage[thicklines]{cancel}
\usepackage{subcaption}
\usepackage[authoryear,semicolon]{natbib}

\usepackage[force,almostfull]{textcomp}
\usepackage{fancyvrb}
\usepackage{textcomp}
\usepackage{url}
\usepackage{ifdraft}
\usepackage{url}
\usepackage{multirow}
\usepackage{rotating}
\usepackage{xspace}
\usepackage{array}
\usepackage{multido}
\usepackage{enumerate}
\usepackage[utf8]{inputenc}
\newlength\figureheight 
\newlength\figurewidth 
\usepackage{graphics}
\usepackage{pgfplots}
\pgfplotsset{compat=newest}
\pgfplotsset{plot coordinates/math parser=false}

\usepackage{scalefnt}

\newtheoremstyle{specialcasestyle}{1mm}{1mm}{\upshape}{}{\bfseries\upshape}{.}{0mm}{}
\theoremstyle{specialcasestyle}

\newtheorem{assump}{Assumption}
\newtheorem{lem}{Lemma}
\newtheorem{corr}{Corollary}

\newtheorem{rem}{Remark}
\newtheorem{thm}{Theorem}

\theoremstyle{remark}          

\newcommand*{\Var}[1]{\ensuremath{\mathrm{Var}\left[#1\right]}}
\newcommand*{\E}[1]{\ensuremath{\mathbb{E}\left[#1\right]}}
\newcommand*{\prob}[1]{\ensuremath{\mathbb{P}\left[#1\right]}}
\newcommand*{\tol}{\ensuremath{\mathrm{TOL}}}

\begin{document} 

\title{Double-Loop Importance Sampling for McKean--Vlasov Stochastic Differential Equation}
\author{Nadhir Ben Rached  \thanks{Department of Statistics, School of Mathematics, University of Leeds, UK ({\tt N.BenRached@leeds.ac.uk}).}, Abdul-Lateef Haji-Ali \thanks{Department of Actuarial Mathematics and Statistics, School of Mathematical and Computer Sciences, Heriot-Watt University, Edinburgh, UK ({\tt A.HajiAli@hw.ac.uk}).}, Shyam Mohan Subbiah Pillai\thanks{Corresponding author; Chair of Mathematics for Uncertainty Quantification, Department of Mathematics, RWTH Aachen University, Aachen, Germany({\tt subbiah@uq.rwth-aachen.de}).} \\
and Ra\'ul Tempone  \thanks{Computer, Electrical and Mathematical Sciences \& Engineering Division (CEMSE), King Abdullah University of Science and Technology (KAUST), Thuwal, Saudi Arabia ({\tt raul.tempone@kaust.edu.sa}). Alexander von Humboldt Professor in Mathematics for Uncertainty Quantification, RWTH Aachen University, Aachen, Germany ({\tt tempone@uq.rwth-aachen.de}).}.
        \thanks{This work was supported by the KAUST Office of Sponsored Research (OSR) under Award No. URF/1/2584-01-01 and the Alexander von Humboldt Foundation. This work was also partially performed as part of the Helmholtz School for Data Science in Life, Earth and Energy (HDS-LEE) and received funding from the Helmholtz Association of German Research Centres. For the purpose of open access, the author has applied a Creative Commons Attribution (CC BY) license to any Author Accepted Manuscript version arising from this submission.}
}
\date{}
\maketitle
\thispagestyle{empty}
\begin{abstract}
This paper investigates Monte Carlo (MC) methods to estimate probabilities of rare events associated with solutions to the $d$-dimensional McKean-–Vlasov stochastic differential equation (MV-SDE). MV-SDEs are usually approximated using a stochastic interacting $P$-particle system, which is a set of $P$ coupled $d$-dimensional stochastic differential equations (SDEs). Importance sampling (IS) is a common technique for reducing high relative variance of MC  estimators of rare-event probabilities. We first derive a zero-variance IS change of measure for the quantity of interest by using stochastic optimal control theory. However, when this change of measure is applied to stochastic particle systems, it yields a $P \times d$-dimensional partial differential control equation (PDE), which is computationally expensive to solve. To address this issue, we use the decoupling approach introduced in \citep{is_mvsde}, generating a $d$-dimensional control PDE for a zero-variance estimator of the decoupled SDE. Based on this approach, we develop a computationally efficient double loop MC (DLMC) estimator. We conduct a comprehensive numerical error and work analysis of the DLMC estimator. As a result, we show  optimal complexity of $\order{\tol_{\mathrm{r}}^{-4}}$ with a significantly reduced constant to achieve a prescribed relative error tolerance $\tol_{\mathrm{r}}$. Subsequently, we propose an adaptive DLMC method combined with IS to numerically estimate rare-event probabilities, substantially reducing relative variance and computational runtimes required to achieve a given  $\tol_{\mathrm{r}}$ compared with standard MC estimators in the absence of IS. Numerical experiments are performed on the Kuramoto model from statistical physics.
\end{abstract}

\section{Introduction}
\label{sec:intro}

This paper investigates Monte Carlo (MC) methods to estimate rare-event probabilities associated with solutions to the McKean-–Vlasov stochastic differential equation (MV-SDE). We develop a computationally efficient MC method to estimate $\E{G(X(T))}$, where $G:\mathbb{R}^d \rightarrow \mathbb{R}$ is a given observable and $\{X(t) \in \mathbb{R}^d: t \in [0,T]\}$ is the solution to the MV-SDE up to a finite terminal time $T$. MV-SDEs are a special class of stochastic differential equations (SDEs), the drift and diffusion coefficients of which depend on the law of the solution itself \citep{mckean_vlasov}. Such SDEs result from the mean-field behavior of stochastic interacting particle systems, commonly used to model various phenomena in pedestrian dynamics~ \citep{pedestrian_app}, collective animal behavior~\citep{animal_app}, oscillator systems~\citep{chemical_app,combustion_app}, biological interactions~\citep{biology_app}, and financial mathematics \citep{finance_app}. Several works have studied the existence and uniqueness of solutions to MV-SDEs ~\citep{mvsde_existence}, well-posedness of the associated Kolmogorov forward and backward partial differential equations (PDEs)~\citep{mvsde_pde,mvsde_smoothpde}, and efficient numerical methods to simulate MV-SDEs for certain classes of drift/diffusion coefficients~\citep{mlmc_mvsde,mvsde_timescheme,mvsde_iterative,mvsde_cubature}.

Owing to the dependence of drift/diffusion on the law of the solution, the MV-SDE is often approximated using a stochastic P-particle system, i.e., a set of $P$ coupled $d$-dimensional It\^o SDEs. Under certain conditions, the stochastic particle system approaches the mean-field limit as the number of particles tends to infinity~\citep{mean_field_limit}, commonly referred to as the propagation of chaos. Time evolution of the particle system's joint probability density is given by the $P \times d$-dimensional Fokker-–Planck PDE, the numerical estimation of which is infeasible. Hence, we propose the use of MC methods by simulating approximate MV-SDE sample paths by using  Euler-–Maruyama time-discretized stochastic particle system for bounded, Lipschitz continuous drift/diffusion coefficients~\citep{euler_mvsde}. Previous studies have investigated MC methods using this numerical scheme for smooth, nonrare observables~\citep{ogawa1992monte,mlmc_mvsde}, and they were able to achieve $\order{\tol^{-4}}$ computational complexity for a prescribed error tolerance, $\tol$. However, the use of naïve MC methods is considerably expensive for rare events, owing to the blowing up of the constant associated with the estimator's computational complexity with the increase in the rarity of the event~\citep{is_general_ref}.

To address the problem of estimating rare-event probabilities, the importance sampling (IS) technique~\citep{is_general_ref} is widely used in the context of SDEs. Recent studies have developed IS schemes using stochastic optimal control theory in various contexts, including stochastic reaction networks~\citep{soc_srn}, and sums of independent random variables in communication systems~\citep{soc_sumrvs}. In particular, several studies formulated the connection between stochastic optimal control and IS for standard SDEs~\citep{is_entropy,is_diffusions,is_cross_entropy}. The derived optimal change of measure is often related to the solution of a Hamilton-Jacobi-Bellman (HJB) equation \citep{hjb_theory}. By employing model-reduction techniques~\citep{is_projection,is_model_reduction} and neural network approaches~\citep{hjb_derivation}, the curse of dimensionality, which is commonly encountered in conventional numerical schemes for solving PDEs, is partially alleviated.

The dependence of drift/diffusion coefficients on the law of the solution  complicates the formulation of the corresponding HJB control PDE for MV-SDEs. This is because the PDE for stochastic particle systems would be $P \times d$-dimensional, and hence infeasible to solve numerically. We overcome this  issue by utilizing a decoupling approach introduced by \citep{is_mvsde}. They defined a decoupled MV-SDE, in which the drift/diffusion coefficients depended on an empirical law that was computed beforehand using a stochastic particle system. This enables the decoupling of the law estimation from the  change of measure for the decoupled MV-SDE. In addition, \citep{is_mvsde} employed large deviations and the Pontryagin principle to obtain a deterministic, time-dependent control that minimizes a proxy for the variance. In the current study, we propose a stochastic optimal control formulation to derive a time- and pathwise-dependent control that produces a zero-variance estimator for the decoupled MV-SDE. This optimal control is obtained by solving a $d$-dimensional control PDE. We list the contributions of this study below.

\begin{itemize} 
\item Stochastic optimal control theory is applied to the decoupled MV-SDE to derive a time- and pathwise-dependent IS control resulting in a zero-variance MC estimator, provided that the sign of the observable does not change. We numerically approximate the solution to a low-dimensional control PDE by using conventional finite difference schemes, obtaining a control that considerably reduces the variance. 
\item A double-loop MC (DLMC) estimator is introduced with IS; this is based on the decoupling approach for MV-SDEs. We first estimate the MV-SDE law by using a stochastic particle system and then define the decoupled MV-SDE conditioned on the empirical law.
\item 	Finally, we provide a detailed analysis of the numerical bias and statistical errors for the DLMC estimator and derive its optimal computational complexity. By combining the proposed DLMC estimator with the IS scheme, an optimal complexity of $\order{\tol_{\mathrm{r}}^{-4}}$ is proven, for a prescribed relative error tolerance $\tol_{\mathrm{r}}$. Additionally, the corresponding constant is substantially reduced, enabling the feasible estimation of rare-event probabilities in the context of MV-SDEs.
\end{itemize}

The remainder of this paper is structured as follows. Section~\ref{sec:prelim} introduces the MV-SDE, associated notation, and emphasizes the need for MC methods to estimate expectations associated with the solution of MV-SDE. We also review the optimal change of measure for IS using stochastic optimal control theory for standard SDEs. Section~\ref{sec:setting} introduces the specific problem and presents discussions of the challenges related to the implementation of IS for MV-SDEs. Section~\ref{sec:dlmc} introduces the decoupling approach, and an optimal control for the decoupled MV-SDE is derived. We also define the DLMC estimator; provide a detailed error and work analysis, and formulate an optimal complexity theorem and an adaptive algorithm to choose optimal parameters for the feasible estimation of rare-event probabilities. Section~\ref{sec:results} applies the proposed approach to the Kuramoto model from statistical physics and provides numerical evidence for the theoretical results.


\section{Preliminaries}
\label{sec:prelim}

\subsection{The McKean--Vlasov stochastic differential equation}
\label{sec:mvsde}

Consider the probability space $\{\Omega,\mathcal{F},\{\mathcal{F}_t\}_{t \geq 0},P\}$, where $\mathcal{F}_t$ is the filtration of a standard Wiener process. For functions $b:\mathbb{R}^d \cross \mathbb{R} \longrightarrow \mathbb{R}^d$, $\sigma:\mathbb{R}^d \cross \mathbb{R} \longrightarrow \mathbb{R}^{d \cross d}$, $\kappa_1: \mathbb{R}^d \cross \mathbb{R}^d \longrightarrow \mathbb{R}$, and $\kappa_2: \mathbb{R}^d \cross \mathbb{R}^d \longrightarrow \mathbb{R}$, we consider the following It\^o SDE for the stochastic process, $X: [0,T] \times \Omega \rightarrow \mathbb{R}^d$.

\begin{empheq}[left=\empheqlbrace, right =,]{equation} 
    \label{eqn:mvsde}
    \begin{alignedat}{2}
    \dd X(t) &= b\left(X(t),\int_{\mathbb{R}^d} \kappa_1 (X(t),x) \mu_t(\dd x)\right) \dd t  \\
    &\qquad + \sigma \left(X(t),\int_{\mathbb{R}^d} \kappa_2 (X(t),x) \mu_t(\dd x)\right) \dd W(t), \quad t>0 \\
    X(0) &= x_0 \sim \mu_0 \in \mathcal{P}(\mathbb{R}^d) , 
    \end{alignedat}
\end{empheq}
where $W:[0,T] \cross \Omega \longrightarrow \mathbb{R}^d$ is a standard $d$-dimensional Wiener process with mutually independent components; $\mu_t \in \mathcal{P}(\mathbb{R}^d)$ is the law of $X(t)$, where $\mathcal{P}(\mathbb{R}^d)$ is the space of probability measures on $\mathbb{R}^d$; and  $x_0 \in \mathbb{R}^d$ is a random initial state with distribution $\mu_0 \in \mathcal{P}(\mathbb{R}^d)$. 

Functions $b(\cdot)$ and $\sigma(\cdot)$ are referred to as the drift and diffusion functions/coefficients, respectively. Existence and uniqueness of solutions to \eqref{eqn:mvsde} can be proved under certain regularity and boundedness of $b,\sigma,\kappa_1,$, and $\kappa_2$~\citep{mvsde_existence,mvsde_soln_theory,mvsde_weak_soln,mean_field_limit}. The time-evolution of the deterministic mean-field law, $\mu_t$, is given by the following Fokker--Planck PDE:


\begin{empheq}[left=\empheqlbrace, right = ,]{alignat=2}
    \label{eqn:fokkerplanck_mvsde}
        &-\frac{\partial \mu(s,x;t,y)}{\partial s} - \sum_{i=1}^d \frac{\partial}{\partial x_i} \left(b_i\left(x,\int_{\mathbb{R}^d} \kappa_1(x,z) \mu(s,z;t,y) \dd z \right) \mu(s,x;t,y)\right) \nonumber \\
        &+ \sum_{i=1}^d \sum_{j=1}^d \frac{1}{2}\frac{\partial^2}{\partial x_i \partial x_j} \Bigg( \Bigg. \sum_{k=1}^d \sigma_{ik} \sigma_{jk} \left(x,\int_{\mathbb{R}^d} \kappa_2(x,z) \mu(s,z;t,y) \dd z \right) \nonumber \\
        &\qquad \mu(s,x;t,y) \Bigg. \Bigg) = 0, \quad (s,x) \in (t,\infty) \cross \mathbb{R}^d  \\
        &\mu(t,x;t,y) = \delta_y(x) ,  \nonumber
\end{empheq}
where $\mu(s,x;t,y)$ is the conditional distribution of $X(s)$, given that $X(t) = y$, and $\delta_y(\cdot)$ is the Dirac measure at point $y$. Equation  \eqref{eqn:fokkerplanck_mvsde} represents a nonlinear PDE due to the dependency of $b$ and $\sigma$ on $\mu(s,x;t,y)$. This is also an integral differential equation with nonlocal terms because the drift and diffusion functions depend on an integral over $\mathbb{R}^d$ with respect to $\mu(s,x;t,y)$. The numerical solution of such a nonlinear integral differential equation up to high tolerances can be cumbersome, particularly in higher dimensions ($d \gg 1$). 

A strong approximation to the solution of the MV-SDE \eqref{eqn:mvsde} is obtained by solving a system of $P$ exchangeable It\^o SDEs, also known as an stochastic interacting particle system with pairwise interaction kernels~\citep{mean_field_limit}. For $p=1, \ldots, P$, process $X^P_p:[0,T] \cross \Omega \rightarrow \mathbb{R}^d$ solves the following SDE:
\begin{empheq}[left=\empheqlbrace, right = ,]{alignat=2}
    \label{eqn:strong_approx_mvsde}
        \dd X^P_p(t) &= b\left(X^P_p(t), \frac{1}{P} \sum_{j=1}^P  \kappa_1(X^P_p(t),X^P_j(t)) \right) \dd t \nonumber \\
        &\qquad + \sigma\left(X^P_p(t), \frac{1}{P} \sum_{j=1}^P \kappa_2(X^P_p(t),X^P_j(t)) \right) \dd W_p(t), \quad t>0 \\
        X^P_p(0) &= (x_0)_p \sim \mu_0 \in \mathcal{P}(\mathbb{R}^d) \nonumber
\end{empheq}
where $\{(x_0)_p\}_{p=1}^P$ are independent and identically distributed (iid) random variables sampled from the initial distribution, $\mu_0$, and $\{W_p:[0,T] \cross \Omega \rightarrow \mathbb{R}^d\}_{p=1}^P$ represents mutually independent $d$-dimensional Wiener processes, which are also independent of $\{(x_0)_p\}_{p=1}^P$. Equation \eqref{eqn:strong_approx_mvsde} approximates the mean-field law, $\mu_t$, in \eqref{eqn:mvsde} by an empirical law based on particles $\{X^P_p\}_{p=1}^P$.
\begin{equation}
    \label{eqn:emp_dist_law}
    \mu_t(\dd x) \approx \mu_t^P(\dd x) = \frac{1}{P} \sum_{j=1}^P \delta_{X^P_j(t)} (\dd x) , 
\end{equation}
where the particles $\{X^P_p\}_{p=1}^P$ are identically distributed, but not mutually independent due to the interaction kernels in the drift and diffusion coefficients. 

Strong convergence of the particle system has been proven for a broad class of drift and diffusion coefficients~\citep{mvsde_strong_conv_1,mvsde_strong_conv_2,mvsde_strong_conv_3}. Given the stochastic interacting particle system \eqref{eqn:strong_approx_mvsde}, we can derive the PDE governing the evolution of the joint probability-density function of $\{X^P_p\}_{p=1}^P$ by using the multidimensional Fokker--Planck equation. For convenience, we write this only for $d=1$. Let $\Tilde{\mu}(t,\mathbf{x})$ be the joint probability-density function, where $\mathbf{x} = [x_1,\ldots,x_P] \in \mathbb{R}^P$. Then, 




\begin{empheq}[left=\empheqlbrace, right = ,]{alignat=2}
    \label{eqn:multid_fp_mvsde}
        &-\frac{\partial \Tilde{\mu}(s,\mathbf{x};t,\mathbf{y})}{\partial s} - \sum_{i=1}^P \frac{\partial}{\partial x_i} \left(b\left(x_i,\frac{1}{P} \sum_{j=1}^P  \kappa_1(x_i,x_j)  \right) \Tilde{\mu}(s,\mathbf{x};t,\mathbf{y})\right) \nonumber \\ 
        &+ \sum_{i=1}^P \frac{1}{2}\frac{\partial^2}{\partial x_i^2 } \Bigg( \Bigg. \sigma^2 \left(x_i, \frac{1}{P} \sum_{j=1}^P \kappa_2(x_i,x_j)\right) 
         \Tilde{\mu}(s,\mathbf{x};t,\mathbf{y}) \Bigg.\Bigg) = 0, \quad (s,\mathbf{x}) \in (t,\infty)\cross \mathbb{R}^P  \\
        &\Tilde{\mu}(t,\mathbf{x};t,\mathbf{y}) = \delta_{\mathbf{y}}(\mathbf{x})  \nonumber
\end{empheq}
where $\Tilde{\mu}(s,\mathbf{x};t,\mathbf{y})$ is the distribution of $\mathbf{X}^P(s) = [X^P_1(s),\ldots,X^P_P(s)]$, given that $\mathbf{X}^P(t) = \mathbf{y}$. 

Equation~\eqref{eqn:multid_fp_mvsde} is linear in $\Tilde{\mu}(t,\mathbf{x})$ and does not have integral dependence, as in~\eqref{eqn:fokkerplanck_mvsde}; however it is $P$-dimensional in the case of $d=1$ and $P \times d$-dimensional in general. Thus, the use of conventional numerical methods to solve \eqref{eqn:multid_fp_mvsde} is infeasible for low error tolerances. This motivates the use of MC methods, which do not suffer from the curse of dimensionality.

\subsubsection{Example: Fully Connected Kuramoto Model for Synchronized Oscillators}
\label{sec:kuramoto}

Herein, we test our methodology on a simple, one-dimensional MV-SDE \eqref{eqn:mvsde} referred to as the Kuramoto model, which is commonly used to describe synchronisation in statistical physics, modeling the behavior of large sets of coupled oscillators as systems of $P$-fully connected, synchronized oscillators. The Kuramoto model is widely applied in several domains, including chemical and biological systems~\citep{chemical_app}, neuroscience~\citep{neuroscience_app}, and oscillating flame dynamics~\citep{combustion_app}. Consider a system of $P$-oscillators with state $\{X^P_p\}_{p=1}^P$. Then, $X_p^P:[0,T] \cross \Omega \rightarrow \mathbb{R}$ satisfies the It\^o SDE:
\begin{empheq}[left=\empheqlbrace, right = ,]{alignat=2}
    \label{eqn:kuramoto_model}
        \dd X^P_p(t) &= \left(\nu_p + \frac{1}{P} \sum_{q=1}^P \sin\left(X^P_p(t) - X^P_q(t)\right)\right) \dd t + \sigma \dd W_p (t) , \quad t>0\\
        X^P_p(0) &= (x_0)_p \sim \mu_0 \in \mathcal{P}(\mathbb{R}) , \nonumber
\end{empheq}
where $\{\nu_p\}_{p=1}^P$ are time-independent iid random variables sampled from a prescribed distribution; diffusion $\sigma \in \mathbb{R}$ is constant; $\{(x_0)_p\}_{p=1}^P$ are iid random variables sampled from a prescribed distribution $\mu_0$; $\{W_p:[0,T] \cross \Omega \rightarrow \mathbb{R}\}_{p=1}^P$ are mutually independent one-dimensional Wiener processes; and $\{\nu_p\}_{p=1}^P, \{(x_0)_p\}_{p=1}^P$, and $\{W_p\}_{p=1}^P$ are independent. This coupled particle system reaches the mean-field limit as the number of oscillators tends to infinity. In this limit, each particle behaves according to the following MV-SDE:
\begin{empheq}[left=\empheqlbrace, right = ,]{equation}
    \label{eqn:kuramoto_mvsde}
    \begin{alignedat}{2}
    \dd X(t) &= \left(\nu + \int_{\mathbb{R}} \sin (X(t)-x) \mu_t (\dd x) \right) \dd t + \sigma \dd W(t), \quad t>0 \\
    X(0) &= x_0 \sim \mu_0 \in \mathcal{P}(\mathbb{R})
    \end{alignedat}
\end{empheq}
where $X(t)$ denotes the state of each particle at time $t$, $\nu$ is a random variable sampled from some prescribed distribution, and  $\mu_t$ is the mean-field law of $X(t)$.


\subsection{IS using Stochastic Optimal Control for SDEs}
\label{sec:is_sde}

First, we develop a framework for an IS scheme that minimizes estimator variance through the stochastic optimal control of standard SDEs. \citep{is_entropy} derived the same optimal control problem based on the variational characterization of thermodynamic free energy, whereas we achieve the same results by posing the optimal IS problem as a stochastic optimal control problem, resulting in a time- and pathwise-dependent control that minimizes the IS estimator variance. Let $Y:[0,T] \times \Omega \rightarrow \mathbb{R}^d$ be the solution to the following standard It\^o SDE:

\begin{empheq}[left=\empheqlbrace, right =,]{alignat=2}
    \label{eqn:sde_defn}
        \dd Y(t) &= b(t,Y(t)) \dd t + \sigma(t,Y(t)) \dd W(t), \quad t \in (0,T] \\
        Y(0) &= y_0, \quad y_0 \in \mathbb{R}^d , \nonumber
\end{empheq}
where $W:[0,T] \cross \Omega \rightarrow \mathbb{R}^d$ is a standard $d$-dimensional Wiener process in probability space $\{\Omega,\mathcal{F},P\}$ with mutually independent components. We aim at estimating $\E{G(Y(T))}$ by using MC for some scalar observable $G: \mathbb{R}^d \longrightarrow \mathbb{R}$, and we apply a change of measure to \eqref{eqn:sde_defn} such that the MC estimator variance is minimized. 

Let us first perform a time-discrete change of measure on the Euler--Maruyama discretization of \eqref{eqn:sde_defn}. We extend this later to the time-continuous setting, obtaining a time-continuous control problem. Consider the discretization $0=t_0<t_1<\ldots<t_N = T$ of the  time domain $[0,T]$ with $N$ uniform time steps; it follows that $t_n = n \times \Delta t, \quad n=0,1,\ldots,N$, and $\Delta t = T/N$. Next, let $Y^N$ be the time-discretized version of process $Y$. Then, the Euler--Maruyama time discretization for the SDE can be expressed as follows:
\begin{empheq}[left=\empheqlbrace, right = \cdot]{alignat=2}
    \label{eqn:sde_euler_re}
    \begin{split}
        &Y^N(t_{n+1}) = Y^N(t_n) + b\left(t_n,Y^N(t_n)\right) \Delta t \\
        & \qquad + \sigma \left(t_n,Y^N(t_n)\right) \sqrt{\Delta t} \epsilon_n,\quad \forall n =0,\ldots,N-1 \\
        &Y^N(t_0) = Y(0) = y_0 , 
    \end{split}
\end{empheq}
where $\epsilon_n \sim \mathcal{N}(0,\mathbbm{I}_d)$ for $n=1,\ldots,N$ are iid random variables; and $\mathbbm{I}_d$ is the $d$-dimensional identity matrix. For a given time step, $n \in \{0,\ldots,N-1\}$, we perform a mean-shift measure change 
\begin{equation}
\label{eqn:is_sde_shift_measure}
    \hat{\epsilon}_n = \sqrt{\Delta t} \zeta_n + \epsilon_n ,
\end{equation}
where $\zeta_n = \zeta(t_n,Y^N(t_n)) \in \mathbb{R}^d$ is a $d$-dimensional control. The resulting likelihood factor $L_n$ is the ratio of two Gaussians and can be written at time step $n$ as 
\begin{align}
\label{eqn:is_sde_discrete_llhood}
    L_n(\hat{\epsilon}_n) &= \exp{-\frac{1}{2} \norm{\hat{\epsilon}_n}^2}\exp{\frac{1}{2}\norm{\hat{\epsilon}_n - \sqrt{\Delta t}\zeta_n)}^2} \nonumber \\
    &= \exp{\frac{1}{2} \Delta t \norm{\zeta_n}^2 - \sqrt{\Delta t} \langle \hat{\epsilon}_n,\zeta_n\rangle } , 
\end{align}
where $\langle\cdot,\cdot\rangle$ is the Euclidean dot product between two vectors in $\mathbb{R}^d$; and $\norm{\cdot}$ is the Euclidean norm for a vector in $\mathbb{R}^d$. Substituting \eqref{eqn:is_sde_shift_measure} in \eqref{eqn:is_sde_discrete_llhood} and setting $L_n$ as a function of $\epsilon_n$, we get
\begin{equation}
\label{eqn:is_sde_discrete_llhood_2}
    L_n(\epsilon_n) = \exp{-\frac{1}{2} \Delta t \norm{\zeta_n}^2 -  \sqrt{\Delta t} \langle\epsilon_n,\zeta_n\rangle} \cdot
\end{equation}

Thus, the likelihood over $N$ time steps is written as $\prod_{n=0}^{N-1} L_n(\epsilon_n)$; hence, the quantity of interest can be expressed as 
\begin{equation}
    \label{eqn:sde_mcis_est}
    \E{G(Y^N(T))} = \E{G(Y^N_\zeta(T)) \prod_{n=0}^{N-1} L_n(\epsilon_n)} ,
\end{equation}
where $Y^N_\zeta$ is subject to 
\begin{empheq}[left=\empheqlbrace, right = \cdot]{alignat=2}
    \label{eqn:sde_euler_is}
    \begin{split}
        &Y^N_\zeta(t_{n+1}) = Y^N_\zeta(t_n) + \Bigg( \Bigg. b\left(t_n,Y^N_\zeta(t_n)\right) 
        + \sigma \left(t_n,Y^N_\zeta(t_n)\right) \zeta(t_n,Y^N_\zeta(t_n)) \Bigg. \Bigg) \Delta t \\
        &\qquad + \sigma \left(t_n,Y^N_\zeta(t_n)\right) \sqrt{\Delta t} \epsilon_n,\quad \forall n =0,\ldots,N-1 \\
        &Y^N_\zeta(t_0) = Y^N(0) = y_0  . 
    \end{split}
\end{empheq}

We aim to minimize the MC estimator variance. Since \eqref{eqn:sde_mcis_est} is an unbiased estimator, it is sufficient to minimize the estimator’s second moment:
\begin{align}
\label{eqn:sde_discrete_optim}
    &\min_{\{\zeta_n\}_{n=1}^N} \E{G^2(Y^N_\zeta(T)) \prod_{n=0}^{N-1} L^2_n(\epsilon_n) \quad \Bigg| \quad
     \mid Y^N_\zeta(0) = y_0} \nonumber \\
    = &\min_{\{\zeta_n\}_{n=1}^N} \E{G^2(Y^N_\zeta(T)) \prod_{n=0}^{N-1} \exp{-\Delta t \norm{\zeta_n}^2 - 2 \langle\epsilon_n,\zeta_n\rangle \sqrt{\Delta t}} \quad \Bigg| \quad Y^N_\zeta(0) = y_0} \cdot
\end{align}

By considering the limit as $\Delta t \rightarrow 0$, we define the cost function 
\begin{equation}
\label{eqn:sde_cost_fxn}
    C_{t,x}(\zeta) = \mathbb{E}\left[G^2(Y_\zeta(T)) \exp{-\int_t^T \norm{\zeta(s,Y_\zeta(s))}^2 \dd s - 2 \int_t^T \langle\zeta(s,Y_\zeta(s)),\dd W(s)\rangle} \quad \Bigg| \quad Y_\zeta(t)=x\right],
\end{equation}
where $Y_\zeta:[0,T] \cross \Omega \rightarrow \mathbb{R}^d$ is subject to 
\begin{empheq}[left=\empheqlbrace, right = \cdot]{alignat=2}
    \label{eqn:sde_sde_is}
    \begin{split}
        \dd Y_\zeta(t) &= \Bigg( \Bigg. b\left(t,Y_\zeta(t) \right) 
        + \sigma \left(t,Y_\zeta(t) \right) \zeta(t,Y_\zeta(t)) \Bigg. \Bigg) \dd t 
        + \sigma \left(t,Y_\zeta(t) \right) \dd W(t), \quad 0<t<T \\ 
        Y_\zeta(0) &= Y(0) = y_0  
    \end{split}
\end{empheq}

Here, $\zeta: [0,T] \cross \mathbb{R}^d \rightarrow \mathbb{R}^d$ is referred to as the IS control. The additional term in the drift ensures that sample paths for $Y_\zeta(t)$ are shifted toward the regime of interest, where the path change is controlled by $\zeta(\cdot,\cdot)$. Hence, $\zeta$ is a control or IS parameter that defines the change of measure, and $Y_\zeta(t) = Y(t)$ only when $\zeta(t,y) = 0, \quad \forall (t,y) \in [0,T] \cross \mathbb{R}^d$. This defined measure change corresponds precisely to the Girsanov theorem for change of measure in SDEs~\citep{sde_oksendal}. \citep{Melnikov2023} also derived the Girsanov theorem via a discrete time formulation, as described above.

Next, we define the value function, which minimizes the second moment of the MC estimator
\begin{equation}
    \label{sde_value_fxn}
    u(t,x) = \min_{\zeta \in \mathcal{Z}} C_{t,x}(\zeta),
\end{equation}
where $\mathcal{Z} = \left\{ f \in C^1 \left( [0,T] \cross \mathbb{R}^d \rightarrow \mathbb{R}^d \right) \right\}$ is the set of admissible deterministic $d$-dimensional Markov controls~\citep{hjb_derivation}.

Subsequently, we solve the optimization problem \eqref{sde_value_fxn} under dynamics \eqref{eqn:sde_sde_is} to determine optimal control $\zeta(\cdot,\cdot)$. In existing literature on control theory, several methods were proposed to solve such minimization problems~\citep{is_entropy,is_cross_entropy}, and the analytical solution for \eqref{sde_value_fxn} is provided by the HJB PDE~ \citep{hjb_derivation}. Herein, we rederive the HJB PDE by using a dynamic programming approach, wherein cost function \eqref{eqn:sde_cost_fxn} possesses a multiplicative structure instead of an additive structure with a certain running plus terminal cost, commonly encountered in optimal control problems. Therefore, we first state and prove a dynamic programming equation for the value function $u$.

\begin{lem}[Dynamic Programming for Standard SDEs]
\label{lemma:dyn_prog}
\hspace{1mm} Let process $Y$ solve the SDE \eqref{eqn:sde_defn}, and controlled process $Y_\zeta$ solve the SDE \eqref{eqn:sde_sde_is}, where  $\zeta:[0,T] \cross \mathbb{R}^d \rightarrow \mathbb{R}^d$ is the control. Assume the value function $u:[0,T] \cross \mathbb{R}^d \rightarrow \mathbb{R}^d$, defined in \eqref{sde_value_fxn}, has bounded and continuous derivatives up to first order in time and second order in space. Then, $u$ satisfies the following dynamic programming relation for all $0<\delta<T-t$,
\begin{equation}
\label{eqn:sde_dyn_prog}
    u(t,x) = \min_{\zeta:[t,t+\delta] \rightarrow \mathbb{R}^d} J_{t,x}(\zeta), \quad \forall x \in \mathbb{R}^d  ,  
\end{equation}
where 
\begin{align}
    J_{t,x}(\zeta) &= \mathbb{E}\Bigg[ \Bigg. \exp{-\int_t^{t+\delta} \norm{\zeta(s,Y_\zeta(s))}^2 \dd s - 2 \int_t^{t+\delta} \langle \zeta(s,Y_\zeta(s)),\dd W(s)\rangle} \nonumber \\ & \qquad u(t+\delta,Y_\zeta(t+\delta)) \quad \Bigg| \quad Y_\zeta(t) = x \Bigg. \Bigg] \cdot
\end{align}
\end{lem}

\begin{proof}
See Appendix~\ref{appendix:a1}.
\end{proof}

From Lemma \ref{lemma:dyn_prog}, we can derive the PDE which solves for  value function \eqref{sde_value_fxn} and subsequently obtain the optimal control, $\zeta$.

\begin{thm}[Optimal Control to Minimize Variance for Standard SDEs]
\label{th:sde_optimal_control}
    \hspace{1mm} Let process $Y$ solve the SDE \eqref{eqn:sde_defn}, and controlled process $Y_\zeta$ solve the SDE \eqref{eqn:sde_sde_is}, where  $\zeta:[0,T] \cross \mathbb{R}^d \rightarrow \mathbb{R}^d$ is the control. In addition, assume that value function $u$ defined in \eqref{sde_value_fxn} has bounded and continuous derivatives up to first order in time and second order in space~\citep{hjb_derivation}, and $u(t,x) \neq 0 \quad \forall (t,x) \in [0,T] \cross \mathbb{R}^d$. Then, $u$ satisfies the following PDE:
    \begin{empheq}[left=\empheqlbrace, right = \cdot]{alignat=2}
    \label{eqn:sde_hjb_form1}
    \begin{split}
        &\frac{\partial u}{\partial t} + \langle b\left(t,x \right), \nabla u \rangle + \frac{1}{2} \nabla^2 u : \left(\sigma \sigma^T\right) \left(t,x \right) \\
        & - \frac{1}{4u} \norm{\sigma^T \nabla u \left(t,x \right)}^2 = 0, \quad (t,x) \in [0,T) \cross \mathbb{R}^d \\
        &u(T,x) = G^2(x), \quad x \in \mathbb{R}^d  , 
    \end{split}
    \end{empheq}
    with an optimal control that minimizes the second moment defined in \eqref{eqn:sde_cost_fxn},
    \begin{equation}
    \label{eqn:sde:hjb_minimizer}
    \zeta^*(t,x) = \frac{1}{2} \sigma^T \left(t,x \right) \nabla \log u (t,x),
    \end{equation}
where $\nabla \cdot$ is the gradient vector, $\nabla^2 \cdot$ is the Hessian matrix for a scalar function, and $\cdot : \cdot$ is the Frobenius inner product between two matrix-valued functions.
\end{thm}

\begin{proof}
See Appendix \ref{appendix:a2}.
\end{proof}

We can obtain an HJB type PDE from \eqref{eqn:sde_hjb_form1} by considering the change of variable $u(t,x) = \exp{-2\gamma(t,x)}$, as shown in Corollary \ref{corr:sde_hjb}.
\begin{corr}[HJB PDE]
\label{corr:sde_hjb}
    \hspace{1mm}Let process $Y$ solve SDE \eqref{eqn:sde_defn}, and controlled process $Y_\zeta$ solve SDE \eqref{eqn:sde_sde_is}, where  $\zeta:[0,T] \cross \mathbb{R}^d \rightarrow \mathbb{R}^d$ is the control. In addition, assume that value function $u$ defined in \eqref{sde_value_fxn} has bounded and continuous derivatives up to first order in time and second order in space, and $u(t,x) \neq 0 \quad \forall (t,x) \in [0,T] \cross \mathbb{R}^d$. Then, $\gamma:[0,T] \cross \mathbb{R}^d \rightarrow \mathbb{R}^d$ satisfies the nonlinear HJB equation,
    \begin{empheq}[left=\empheqlbrace, right = \cdot]{alignat=2}
    \label{eqn:sde_hjb_form2}
    \begin{split}
        &\frac{\partial \gamma}{\partial t} + \langle b\left(t,x \right), \nabla \gamma \rangle + \frac{1}{2} \nabla^2 \gamma : \left(\sigma \sigma^T\right) \left(t,x \right) \\
        &- \frac{1}{2} \norm{\sigma^T \nabla \gamma \left(t,x\right)}^2 = 0, \quad (t,x) \in [0,T) \cross \mathbb{R}^d \\
        &\gamma(T,x) = - \log \abs{G(x)}, \quad x \in \mathbb{R}^d  , 
    \end{split}
    \end{empheq}
    with optimal control 
    \begin{equation}
        \zeta^*(t,x) = - \sigma^T (t,x) \nabla \gamma \left(t,x \right),
    \end{equation}
    which minimizes the second moment defined in \eqref{eqn:sde_cost_fxn}.
\end{corr}


\begin{rem}[Existence and uniqueness of HJB solutions]
	\hspace{1mm} Classical HJB solution theory requires that $\gamma$ has bounded and continuous derivatives up to first order in time and second order in space~\citep{hjb_derivation}. However, these assumptions can be relaxed by introducing the notion of viscosity solutions~\citep{Soner1997}. Since Corollary~\ref{corr:sde_hjb} is a well-known classical result, which we alternatively derive via an optimal control approach, we refer the readers to~\citep{Pham2009,hjb_derivation} for a more formal proof. 
\end{rem}

\begin{rem}[Equivalent HJB Formulation]
    \hspace{1mm}Previous approaches formulated the same optimal control for IS, but through variational characterization of thermodynamic free energy~\citep{is_entropy,hjb_derivation}, minimizing the following value function
    \begin{equation}
        \label{eqn:dvmsde_is_entropy}
        \gamma(t,x) = \inf_{\zeta \in \mathcal{Z}} \E{\int_t^T \frac{1}{2} \norm{\zeta(s,Y_\zeta(s))}^2 \dd s - \log (G(Y_\zeta(T))) \quad \Bigg| \quad Y_\zeta(t) = x},
    \end{equation}
where $Y_\zeta$ follows dynamics \eqref{eqn:sde_sde_is}, which leads to the HJB equation \eqref{eqn:sde_hjb_form2}. 
\end{rem}

Various approaches have been proposed to numerically solve \eqref{eqn:sde_hjb_form2} and obtain an approximate control. \citep{is_regression} solved the $d$-dimensional HJB PDE \eqref{eqn:sde_hjb_form2} using least-squares regression, whereas~\citep{is_model_reduction} solved it using model-reduction techniques for higher dimensions. Neural networks have also been employed to solve the HJB PDE in higher dimensions with stochastic gradient~\citep{is_entropy} and cross-entropy~\citep{is_cross_entropy} learning methods for the stochastic optimal control formulation \eqref{eqn:dvmsde_is_entropy}. In contrast, we use the equivalency of nonlinear HJB and linear Kolmogorov backward equation (KBE). It is worth recalling from previous discussions that $u$ is the value function that minimizes the second moment. The proposed change of measure with optimal control in Corollary \ref{corr:sde_hjb} produces a zero-variance estimator, provided that $G(\cdot)$ does not change sign. This can be seen by substituting $u(t,x)=v^2 (t,x)$ in \eqref{eqn:sde_hjb_form1} to obtain a PDE for $v:[0,T] \cross \mathbb{R}^d \rightarrow \mathbb{R}^d$:

\begin{empheq}[left=\empheqlbrace, right =,]{alignat=2}
    \label{eqn:sde_hjb_form3}
    \begin{split}
        &\frac{\partial v}{\partial t} + \langle b\left(t,x \right), \nabla v \rangle
        + \frac{1}{2} \nabla^2 v : \left(\sigma \sigma^T\right) \left(t,x \right) = 0 ,\quad (t,x) \in [0,T) \cross \mathbb{R}^d \\
        &v(T,x) = \abs{G(x)}, \quad x \in \mathbb{R}^d
    \end{split}
\end{empheq}
with optimal control 
\begin{equation}
    \label{eqn:sde_hjb_optimal_control}
    \zeta^*(t,x) = \sigma^T \left(t,x \right) \nabla \log v (t,x) \cdot
\end{equation}

As \eqref{eqn:sde_hjb_form3} is the KBE that solves for the conditional expectation $\E{\abs{G(Y(T))} \mid Y(t)=x}$, where process $Y$ follows dynamics \eqref{eqn:sde_defn}, the second moment is equal to the square of the first moment, hence leading to zero variance, provided that the sign of  $G(\cdot)$ remains constant. Thus, solving the linear KBE \eqref{eqn:sde_hjb_form3} is sufficient for obtaining an optimal control for the zero variance estimator, provided $G(\cdot)$ does not change sign. For IS purposes, roughly solving \eqref{eqn:sde_hjb_form3} is sufficient for obtaining substantial variance reduction. Hence, we employ MC methods with IS by employing a control derived by numerically solving \eqref{eqn:sde_hjb_form3} to estimate rare event probabilities. This is computationally much cheaper than directly solving the KBE to estimate the quantity of interest to achieve relative error tolerances. However, it would still be expensive to solve the multidimensional KBE when $d \gg 1$. This work does not consider high-dimensional problems, which are left for future work.

\begin{rem}[Using KBE for IS]
	\hspace{1mm} Derivation of a zero variance estimator that involves the solution to the KBE is a classical result~\citep{Newton1994}. As IS produces an unbiased estimator, it is sufficient to roughly approximate $v$ and its derivatives to get an approximate but useful control. \citep{hinds2023neural} applied this concept within the context of standard SDEs. Herein, we apply this in the context of MV-SDEs. 
\end{rem}

\section{Problem Setting}
\label{sec:setting}

Let $T>0$ be some finite terminal time and $X$ be the solution to the MV-SDE \eqref{eqn:mvsde}. Let $G: \mathbb{R}^d \longrightarrow \mathbb{R}$ be a given scalar observable function. Our objective is to build a computationally efficient estimator, $\mathcal{A}$, of $\E{G(X(T))}$ for some specified relative tolerance $\tol_{\mathrm{r}} > 0$ that satisfies the following:
\begin{equation}
    \label{eqn:mc_objective}
    \prob{\frac{\abs{\mathcal{A}-\E{G(X(T))}}}{\abs{\E{G(X(T))}}} < \tol_{\mathrm{r}}} > 1-\alpha,
\end{equation}
where $0 < \alpha \ll 1$ determines the confidence level. 



The existence and uniqueness of solutions to the corresponding KBE are not straightforward problems; this is a consequence of drift/diffusion dependencies on the law $\mu_t$~\citep{kbe_mvsde}. This problem can be circumvented by formulating the KBE for the stochastic $P$-particle system \eqref{eqn:strong_approx_mvsde}. Computational costs for numerically solving the corresponding $P \cross d$-dimensional KBE for a given $\tol_{\mathrm{r}}$ scales exponentially with the dimension. In addition, conventional numerical schemes are not equipped to handle relative error tolerances, complicating error control even in one-dimensional problems. We overcome this issue by employing MC methods. In the context of rare events, the feasibility of naive MC rapidly diminishes, as the number of sample paths required to satisfy a given statistical error tolerance scales inversely with the probability of the event to be estimated. Thus, we combine IS with MC methods as a variance-reduction technique, producing computationally feasible estimates of rare-event probabilities. Section \ref{sec:is_sde} presented the method to obtain an optimal IS change of measure for standard SDEs. However, there are two main difficulties while solving the variance minimization problem for MV-SDEs.

\begin{enumerate}
    \item Deriving optimal change of measure for MV-SDEs is not straightforward, due to the dependency of drift and diffusion functions on $\mu_t$. Any change of measure will result in changes in drift $b(\cdot,\cdot)$ and diffusion $\sigma(\cdot,\cdot)$ coefficients. 
    \item Suppose we use stochastic $P$-particle approximation \eqref{eqn:strong_approx_mvsde} and wish to apply a change of measure on it. Solving a variance minimization problem in this context will result in an HJB PDE in $P \times d$ dimensions. Numerical methods to solve such an equation would suffer from the curse of dimensionality when $P \times d \gg 1$.
\end{enumerate}

Therefore, we use a decoupling approach~\citep{is_mvsde} instead of  considering a change of measure on the stochastic particle system or the MV-SDE itself.

\section{DLMC with IS}
\label{sec:dlmc}

\subsection{Decoupling Approach for MV-SDEs}

The decoupling approach was introduced by \citep{is_mvsde} as a method to efficiently implement IS for MV-SDEs. This method involves replacing the deterministic mean-field law ($\mu_t$) with an empirical approximation using a stochastic particle system, which is then used as an input to define the "decoupled" MV-SDE to which a change in measure is applied. This decouples the computation of the MV-SDE law and the probability measure change required for IS. The decoupled MV-SDE \eqref{eqn:mvsde} is now a standard SDE with random coefficients, for which change of measure can be applied, as discussed in Section \ref{sec:is_sde}. First, we formally introduce the general scheme of the decoupling approach.

\begin{enumerate}
    \item As we do not have direct access to $\{\mu_t: t \in [0,T]\}$ for the MV-SDE, we approximate it by using the empirical measure, $\{\mu_t^P: t \in [0,T]\}$, from \eqref{eqn:emp_dist_law} with particles $\{X^P_p(t): t \in [0,T]\}_{p=1}^P$ obtained from the stochastic interacting particle system~\eqref{eqn:strong_approx_mvsde}. 
    \item Given the empirical law, $\{\mu_t^P: t \in [0,T]\}$, we define the "decoupled" MV-SDE process, $\Bar{X}^P:[0,T] \times \Omega \rightarrow \mathbb{R}^d$, following the dynamics described as follows:
    \begin{empheq}[left=\empheqlbrace, right =,]{alignat=2}
        \label{eqn:decoupled_mvsde}
        \begin{split}
            &\dd \Bar{X}^P(t) = b\left(\Bar{X}^P(t), \frac{1}{P} \sum_{j=1}^P  \kappa_1(\Bar{X}^P(t),X^P_j(t)) \right) \dd t \\
            &\qquad + \sigma \left(\Bar{X}^P(t), \frac{1}{P} \sum_{j=1}^P  \kappa_2(\Bar{X}^P(t),X^P_j(t)) \right) \dd \Bar{W}(t), \quad t \in [0,T] \\
            &\Bar{X}^P(0) = \Bar{x}_0 \sim \mu_0, \quad \Bar{x}_0 \in \mathbb{R}^d , 
        \end{split}
    \end{empheq}
where the $P$ in $\Bar{X}^P(t)$ indicates that the drift and diffusion functions in \eqref{eqn:decoupled_mvsde} are computed using empirical law $\{\mu_t^P: t \in [0,T]\}$ derived from a $P$-particle system. In addition, the  drift and diffusion coefficients $b$ and $\sigma$ are the same as those  defined in Section~ \ref{sec:mvsde}; $\Bar{W}:[0,T] \cross \Omega \rightarrow \mathbb{R}^d$ is a standard $d$-dimensional Wiener process that is independent of the Wiener processes $\{W_p\}_{p=1}^P$ used in \eqref{eqn:strong_approx_mvsde}; and $\Bar{x}_0 \in \mathbb{R}^d$ is a random initial state sampled from the $\mu_0$ distribution as defined in \eqref{eqn:mvsde} and is independent from $\{(x_0)_p\}_{p=1}^P$ used in \eqref{eqn:strong_approx_mvsde}. Equation \eqref{eqn:decoupled_mvsde} is a standard SDE with random coefficients, whose randomness arises from drift and diffusion dependencies on the empirical measure $\{\mu_t^P: t \in [0,T]\}$. The existence and uniqueness of a solution to such SDEs have been shown previously~\citep{random_sde_theory}.
    \item We introduce a copy space to distinguish \eqref{eqn:decoupled_mvsde} from the stochastic $P$-particle system~\citep{is_mvsde}. Suppose \eqref{eqn:strong_approx_mvsde} is defined on the probability space $(\Omega,\mathcal{F},\mathbb{P})$. We define a copy space $(\Bar{\Omega},\Bar{\mathcal{F}},\Bar{\mathbb{P}})$, and \eqref{eqn:decoupled_mvsde} is defined on the product space $(\Omega,\mathcal{F},\mathbb{P}) \cross (\Bar{\Omega},\Bar{\mathcal{F}},\Bar{\mathbb{P}})$. Thus, $\mathbb{P}$ is a probability measure induced by the randomness of $\{\mu_t^P: t \in [0,T]\}$, and $\Bar{\mathbb{P}}$ is the  measure induced by the randomness driving the decoupled MV-SDE dynamics  \eqref{eqn:decoupled_mvsde}, conditioned on $\{\mu_t^P: t \in [0,T]\}$. 
    \item Thus, we approximate our quantity of interest as 
    \begin{align}
    \label{eqn:total_exp_mvsde}
        \E{G(X(T))} & \approx \mathbb{E}_{\mathbb{P} \otimes \Bar{\mathbb{P}}} \left[G(\Bar{X}^P(T))\right] \nonumber \\
        &= \mathbb{E}_\mathbb{P} \left[\mathbb{E}_{\Bar{\mathbb{P}}} \left[ G(\Bar{X}^P(T)) \mid \{\mu_t^P: t \in [0,T]\} \right]\right] ,
    \end{align}
Henceforth, for ease of notation, $\E{G(\Bar{X}^P(T))}$ indicates that the  expectation is taken with respect to all sources of randomness in the decoupled MV-SDE \eqref{eqn:decoupled_mvsde}. First, we estimate the inner expectation $\E{G(\Bar{X}^P(T))\mid \{\mu_t^P: t \in [0,T]\}}$, and then we estimate the outer expectation by using MC sampling over different realizations of the empirical law.
\end{enumerate}



The inner expectation $\E{G(\Bar{X}^P(T))\mid \{\mu_t^P: t \in [0,T]\}}$ solves the KBE associated with \eqref{eqn:decoupled_mvsde}. However, obtaining an analytical solution is not always possible, and conventional numerical methods cannot handle relative error tolerances even for $d=1$, which is relevant for rare events. Even if the KBE could be solved accurately, it would need to be solved multiple times for each empirical law realization, and this could prove cumbersome. Therefore, we propose to use MC methods coupled with IS, even for the one-dimensional case, to estimate the nested expectation. 
\begin{rem}[Time Extension of the Empirical Law]
\hspace{1mm} In practice, we only have access to a time-discretized $\{\mu_t^P: t \in [0,T]\}$ from the Euler--Maruyama time discretization of  \eqref{eqn:strong_approx_mvsde}. However, $\mu^P_t$ must be defined continuously throughout the time domain for the decoupled MV-SDE to be well-defined. Therefore, we use the continuous Euler time extension for the discretized empirical law over the entire time domain.
\end{rem}

\subsection{Optimal IS for Decoupled MV-SDE using Stochastic Optimal Control}

We introduced the decoupled MV-SDE to implement an IS change of measure for a  given $\{\mu_t^P: t \in [0,T]\}$ on a standard lower-dimensional SDE, thereby  reducing MC estimator variance for rare-event probabilities. As \eqref{eqn:decoupled_mvsde} is a standard SDE for a fixed $\{\mu_t^P:t \in [0,T]\}$, we follow the procedure in Section~\ref{sec:is_sde} to derive the  optimal change of measure.

\begin{corr}[HJB PDE for decoupled MV-SDE]
\hspace{1mm} Let the decoupled process $\Bar{X}^P$ follow the dynamics \eqref{eqn:decoupled_mvsde} and let the controlled process $\Bar{X}^P_\zeta:[0,T] \cross \Omega \rightarrow \mathbb{R}^d$ follow the given controlled dynamics with control $\zeta:[0,T] \cross \mathbb{R}^d \rightarrow \mathbb{R}^d$:
    \begin{empheq}[left=\empheqlbrace, right = \cdot]{alignat=2}
    \label{eqn:dmvsde_sde_is}
      \begin{split}
        \dd \Bar{X}^P_\zeta(t) &= \Bigg( \Bigg. b\left(\Bar{X}^P_\zeta(t), \frac{1}{P} \sum_{j=1}^P  \kappa_1(\Bar{X}^P_\zeta(t),X^P_j(t)) \right) \\
        &+ \sigma \left(\Bar{X}^P_\zeta(t), \frac{1}{P} \sum_{j=1}^P  \kappa_2(\Bar{X}^P_\zeta(t),X^P_j(t)) \right) \zeta(t,\Bar{X}^P_\zeta(t)) \Bigg. \Bigg) \dd t \\
        &+ \sigma \left(\Bar{X}^P_\zeta(t), \frac{1}{P} \sum_{j=1}^P  \kappa_2(\Bar{X}^P_\zeta(t),X^P_j(t)) \right) \dd W(t), \quad 0<t<T \\ 
         \Bar{X}^P_\zeta(0) &= \Bar{X}^P(0) = \Bar{x}_0 \sim \mu_0 . 
        \end{split}
    \end{empheq}
    Assuming that \eqref{eqn:strong_approx_mvsde} is used to compute $\{\mu^P_t:t \in [0,T]\}$ in \eqref{eqn:decoupled_mvsde} and \eqref{eqn:dmvsde_sde_is}, the value function $u:[0,T] \cross \mathbb{R}^d \rightarrow \mathbb{R}^d$ that minimizes the second moment can be defined as 
    \begin{align}
    \label{eqn:dmvsde_value_fxn}
        u(t,x) &= \min_{\zeta \in \mathcal{Z}}      \mathbb{E}\Bigg[\Bigg.G^2(\Bar{X}_\zeta^P(T)) \exp{-\int_t^T \norm{\zeta(s,\Bar{X}_\zeta^P(s))}^2 - 2 \int_t^T \langle \zeta(s,\Bar{X}_\zeta^P(s)),\dd W(s) \rangle} \nonumber\\
        &\qquad \mid \Bar{X}_\zeta^P(t) = x, \{\mu_t^P: t \in [0,T]\}\Bigg.\Bigg] \cdot
    \end{align} 
    Here, $\mathcal{Z} = \left\{ f \in C^1 \left( [0,T] \cross \mathbb{R}^d \rightarrow \mathbb{R}^d \right) \right\}$ is a set of admissible deterministic $d$-dimensional Markov controls. Next, assume that $u$ has bounded and continuous derivatives up to first order in time and second order in space, and $u(t,x) \neq 0 \quad \forall (t,x) \in \mathbb{R}^d$. Define a new function $\gamma:[0,T] \cross \mathbb{R}^d \rightarrow \mathbb{R}^d$, such that
    \begin{equation}
        u(t,x) = \exp{-2 \gamma(t,x)} \cdot
    \end{equation}
    Then, $\gamma$ satisfies the following nonlinear HJB equation:
    \begin{empheq}[left=\empheqlbrace, right =,]{alignat=2}
    \label{eqn:dmvsde_hjb_form2}
    \begin{split}
        &\frac{\partial \gamma}{\partial t} + \langle b\left(x, \frac{1}{P} \sum_{j=1}^P  \kappa_1(x,X^P_j(t)) \right), \nabla \gamma \rangle + \frac{1}{2} \nabla^2 \gamma : \left(\sigma \sigma^T\right) \left(x, \frac{1}{P} \sum_{j=1}^P  \kappa_2(x,X^P_j(t)) \right)\\
        & - \frac{1}{4} \norm{\sigma^T \nabla \gamma \left(x, \frac{1}{P} \sum_{j=1}^P  \kappa_2(x,X^P_j(t)) \right)}^2 = 0, \quad (t,x) \in [0,T) \cross \mathbb{R}^d \\
        &\gamma(T,x) = - \log \abs{G(x)}, \quad x \in \mathbb{R}^d , 
    \end{split}
    \end{empheq}
    with optimal control 
    \begin{equation}
        \zeta^*(t,x) = - \sigma^T \left(x, \frac{1}{P} \sum_{j=1}^P  \kappa_2(x,X^P_j(t)) \right) \nabla \gamma \left(t,x \right),
    \end{equation}
    which minimizes the second moment \eqref{eqn:dmvsde_value_fxn} conditioned on $\{\mu_t^P: t \in [0,T]\}$.
\end{corr}

The HJB PDE that solves for $u$ results in a zero-variance estimator for a given empirical law, $\mu^P_t$, provided $G(\cdot)$ does not change sign. Hence, we can recover the linear KBE, as shown in Section~\ref{sec:is_sde}, and obtain the following control PDE in the MV-SDE context, 
\begin{empheq}[left=\empheqlbrace, right =,]{alignat=2}
    \label{eqn:dmvsde_hjb_form3}
    \begin{split}
        &\frac{\partial v}{\partial t} + \langle b\left(x, \frac{1}{P} \sum_{j=1}^P  \kappa_1(x,X^P_j(t)) \right), \nabla v \rangle \\
        & \qquad + \frac{1}{2} \nabla^2 v : \left(\sigma \sigma^T\right) \left(x, \frac{1}{P} \sum_{j=1}^P  \kappa_2(x,X^P_j(t)) \right) = 0 ,\quad (t,x) \in [0,T) \cross \mathbb{R}^d \\
        &v(T,x) = \abs{G(x)}, \quad x \in \mathbb{R}^d , 
    \end{split}
\end{empheq}
with optimal control 
\begin{equation}
    \label{eqn:dmvsde_hjb_optimal_control}
    \zeta^*(t,x) = \sigma^T \left(x, \frac{1}{P} \sum_{j=1}^P  \kappa_2(x,X^P_j(t)) \right) \nabla \log v (t,x) \cdot
\end{equation}

\begin{rem}[Zero-Variance Control for MV-SDEs]
\label{rem:zero_variance}
	\hspace{1mm} Let us assume that the deterministic mean-field law ($\mu_t$) is given beforehand. If we condition the decoupled MV-SDE~\eqref{eqn:decoupled_mvsde} on $\{\mu_t:t \in [0,T]\}$, the controls given by~\eqref{eqn:dmvsde_hjb_form3} and~\eqref{eqn:dmvsde_hjb_optimal_control} result in a zero-variance estimator of $\E{G(X(T))}$ for MV-SDEs, given that the sign of $G(\cdot)$ remains unchanged. However, in practice, we only approximate $\mu_t$ by an empirical estimate, resulting in the DLMC estimator.
\end{rem}

Herein, we numerically solve the one-dimensional ($d=1$) KBE corresponding to the Kuramoto model \eqref{eqn:kuramoto_model} numerically using finite differences. Model reduction techniques~\citep{is_model_reduction,is_projection} or the solving of minimization problem \eqref{eqn:dmvsde_value_fxn} using stochastic gradient methods~\citep{is_entropy} are appropriate for higher-dimensional ($d \gg 1$) problems. Due to the lack of closed-form solutions for the inner or outer expectations, obtaining numerical KBE-based solutions for both expectations that satisfy relative error tolerance is infeasible, even in the one-dimensional case. This necessitates the use of a DLMC estimator in this context. The DLMC estimator is constructed through the following steps:

\begin{enumerate}
    \item Approximate the mean-field law $\mu_t$ in \eqref{eqn:mvsde} by the empirical $\mu_t^P$ using \eqref{eqn:strong_approx_mvsde}. In practice, we obtain the discretized empirical law with $N_1$ time steps from Euler--Maruyama time discretization of the particle system. Consider the discretization $0=t_0<t_1<t_2<\ldots<t_{N_1} = T$ of the time domain $[0,T]$ with $N_1$ equal time steps. Hence, $t_n = n \times \Delta t_1, \quad n=0,1,\ldots,N_1$; and $\Delta t_1 = T/{N_1}$. Let us denote by $X_p^{P|N_1}$ the time-discretized version for state $X_p^P$ corresponding to \eqref{eqn:strong_approx_mvsde}. 
    \item Let $\mu^{P|N_1}$ be the discrete law obtained from the above-mentioned step, then
    \begin{equation}
    \label{eqn:dmvsde_discrete_law}
        \mu^{P|N_1}(t_n) = \frac{1}{P} \sum_{p=1}^P \delta_{X_p^{P|N_1}(t_n)}, \quad \forall n=0,\ldots,N_1 \cdot 
    \end{equation}

In order to define a time-continuous extension for the empirical law, we extend the time-discrete stochastic particle system to all $t \in [0,T]$ by  using the continuous-time Euler extension. Given $\{X_p^{P|N_1}(t_n)\}_{p=1}^P$ for all time steps $n=0,\ldots,N_1$, we have 
    \begin{align}
        \label{sps_euler_extension}
        &X_p^{P|N_1}(t) = X_p^{P|N_1}(t_n) + b\left(X_p^{P|N_1}(t_n), \frac{1}{P} \sum_{j=1}^P  \kappa_1(X_p^{P|N_1}(t_n),X_j^{P|N_1}(t_n))\right) (t - t_n) \\
        &+ \sigma \left(X_p^{P|N_1}(t_n), \frac{1}{P} \sum_{j=1}^P  \kappa_2(X_p^{P|N_1}(t_n),X_j^{P|N_1}(t_n))\right) (W(t) - W(t_n)), \quad t_n < t < t_{n+1} \cdot \nonumber
    \end{align}
    The time-continuous empirical law is then defined as follows
    \begin{equation}
    \label{eqn:law_euler_extension}
        \mu^{P|N_1}(t) = \frac{1}{P} \sum_{j=1}^P \delta_{X_j^{P|N_1}(t)}, \quad \forall t \in [0,T] \cdot
    \end{equation}
    
    \item Given $\mu^{P|N_1}$ from \eqref{eqn:dmvsde_discrete_law}, we use \eqref{eqn:dmvsde_hjb_form3} and \eqref{eqn:dmvsde_hjb_optimal_control} to obtain the optimal control $\zeta(\cdot,\cdot)$ for IS.
    \item Given $\mu^{P|N_1}$ from \eqref{eqn:dmvsde_discrete_law} and control $\zeta:[0,T] \times \mathbb{R}^d \rightarrow \mathbb{R}^d$, we define the controlled decoupled process $\Bar{X}^{P|N_1}_\zeta$. Its dynamics can be derived by slightly modifying \eqref{eqn:dmvsde_sde_is}, indicating that the drift and diffusion coefficients depend on the time extended version of $\mu^{P|N_1}$ obtained using $P$ particles and $N_1$ time steps:
    \begin{empheq}[left=\empheqlbrace, right = \cdot]{alignat=2}
    \label{eqn:dmvsde_sde_is_form2}
    \begin{split}
        &\dd \Bar{X}^{P|N_1}_\zeta(t) = \Bigg( \Bigg. b\left(\Bar{X}^{P|N_1}_\zeta(t), \frac{1}{P} \sum_{j=1}^P  \kappa_1(\Bar{X}^{P|N_1}_\zeta(t),X^{P|N_1}_j(t)) \right) \\
        &+ \sigma \left(\Bar{X}^{P|N_1}_\zeta(t), \frac{1}{P} \sum_{j=1}^P  \kappa_2(\Bar{X}^{P|N_1}_\zeta(t),X^{P|N_1}_j(t)) \right) \zeta(t,\Bar{X}^{P|N_1}_\zeta(t)) \Bigg. \Bigg) \dd t \\
        &+ \sigma \left(\Bar{X}^{P|N_1}_\zeta(t), \frac{1}{P} \sum_{j=1}^P  \kappa_2(\Bar{X}^{P|N_1}_\zeta(t),X^{P|N_1}_j(t)) \right) \dd W(t), \quad 0<t<T \\ 
        &\Bar{X}^{P|N_1}_\zeta(0) = \Bar{x}_0 \sim \mu_0 . 
    \end{split}
    \end{empheq}
    \item We use the Euler--Maruyama time discretization with $N_2$ time steps. In our notation, we indicate the three discretization parameters used, namely, $P, N_1$, and $N_2$. Consider a new discretization, $0=\Bar{t}_0<\Bar{t}_1<\Bar{t}_2<\ldots<\Bar{t}_{N_2} = T$, of the  time domain $[0,T]$ with $N_2$ uniform time steps. Hence, $\Bar{t}_n = n \times \Delta t_2, \quad n=0,1,\ldots,N_2$, and $\Delta t_2 = \frac{T}{N_2}$.
    \begin{itemize}
        \item For time step $n=0$, 
        \begin{equation}
            \Bar{X}^{P|N_1|N_2}_\zeta(\Bar{t}_0) = \Bar{x}_0 \cdot \nonumber
        \end{equation}
        \item For $n=1,\ldots,N_2-1$
        \begin{align}
            &\Bar{X}^{P|N_1|N_2}_\zeta(\Bar{t}_{n+1}) = \Bar{X}^{P|N_1|N_2}_\zeta(\Bar{t}_n) + \Bigg( \Bigg. b\left(\Bar{X}^{P|N_1|N_2}_\zeta(\Bar{t}_n), \frac{1}{P} \sum_{j=1}^P  \kappa_1 (\Bar{X}^{P|N_1|N_2}_\zeta(\Bar{t}_n),X_j^{P|N_1}(\Bar{t}_n)) \right) \nonumber \\
            &+ \sigma \left(\Bar{X}^{P|N_1|N_2}_\zeta(\Bar{t}_n), \frac{1}{P} \sum_{j=1}^P  \kappa_2(\Bar{X}^{P|N_1|N_2}_\zeta(\Bar{t}_n),X_j^{P|N_1}(\Bar{t}_n))\right) \zeta(\Bar{t}_n,\Bar{X}^{P|N_1|N_2}_\zeta(\Bar{t}_n)) \Bigg. \Bigg) \Delta t_2 \nonumber \\
            \label{eqn:dmvsde_euler_is_form2}
            &+ \sigma \left(\Bar{X}^{P|N_1|N_2}_\zeta(\Bar{t}_n), \frac{1}{P} \sum_{j=1}^P  \kappa_2(\Bar{X}^{P|N_1|N_2}_\zeta(\Bar{t}_n), X_j^{P|N_1}(\Bar{t}_n)) \right) \Delta \bar{W}_n, 
        \end{align}
        where $\Delta \bar{W}_n \sim \mathcal{N}(0,\sqrt{\Delta t_2} \mathbb{I}_d)$.
        \end{itemize}
        \item Given $\{\Bar{X}^{P|N_1|N_2}_\zeta(\Bar{t}_n)\}_{n=1}^{N_2}$, we can express the quantity of interest with IS as 
        \begin{equation}
        \label{eqn:dmvsde_qoi_is}
            \E{G(\Bar{X}^{P|N_1|N_2}(T))} = \E{G(\Bar{X}_\zeta^{P|N_1|N_2}(T))\mathbb{L}^{P|N_1|N_2}},
        \end{equation}
        where, similar to \eqref{eqn:is_sde_discrete_llhood_2}, the likelihood factor can be expressed as 
        \begin{equation}
        \label{eqn:dlmc_llhood_factor}
            \mathbb{L}^{P|N_1|N_2} = \prod_{n=0}^{N_2-1} \exp{-\frac{1}{2} \Delta t_2 \norm{\zeta(\Bar{t}_n,\Bar{X}_\zeta^{P|N_1|N_2}(\Bar{t}_n))}^2 - \langle \Delta \bar{W}_n , \zeta(\Bar{t}_n,\Bar{X}^{P|N_1|N_2}_\zeta(\Bar{t}_n)) \rangle} \cdot
        \end{equation}
Observe that $\mathbb{L}^{P|N_1|N_2}=1$ when $\zeta(t,x)=0, \quad \forall (t,x) \in [0,T] \times \mathbb{R}^d$.
\end{enumerate}

We propose a general DLMC algorithm (Algorithm~\ref{alg:dlmc_general}) to estimate the required quantity of interest by using Steps 1-6. The output of Algorithm \ref{alg:dlmc_general} is our DLMC estimator, which is defined as follows: 
\begin{equation}
\label{eqn:dmvsde_dlmc_est}
    \mathcal{A}_{\mathrm{MC}} = \frac{1}{M_1} \sum_{m_1=1}^{M_1} \frac{1}{M_2} \sum_{m_2=1}^{M_2} G\left(\Bar{X}_\zeta^{P|N_1|N_2}(T) \right) \mathbb{L}^{P|N_1|N_2} \left( \omega_{1:P}^{(m_1)}, \tilde{\omega}^{(m_2)} \right) \cdot
\end{equation}

In~\eqref{eqn:dmvsde_dlmc_est}, notation $\omega_{1:P}^{(m_1)}$ denotes the $m_1^{\mathrm{th}}$ realization of the $P$ sets of random variables (Wiener increments and initial states) that are used in calculating $\{X_p^{P \mid N_1}(t_n)\}_{n=1}^{N_1}$ for all particles $p=1,\ldots,P$. $\tilde{\omega}^{(m_2)}$ denotes the $m_2^{\mathrm{th}}$ realization of random variables used in calculating $\{\bar{X}_\zeta^{P \mid N_1 \mid N_2}(\bar{t}_n)\}_{n=1}^{N_2}$, given $\mu^{P|N_1}$.

\begin{algorithm}
    \caption{General DLMC algorithm for decoupled MV-SDE}
\label{alg:dlmc_general}
    \SetAlgoLined
    \textbf{Inputs: } $P,N_1,N_2,M_1,M_2$; \\
    \For{$m_1=1,\ldots,M_1$}{
    Generate realization of random variables $\omega_{1:P}^{(m_1)}$; \\
    Generate realization of law $\mu^{P|N_1} \left( \omega_{1:P}^{(m_1)} \right)$ with $P$-particle system and $N_1$ time steps using \eqref{eqn:dmvsde_discrete_law}; \\
    Given $\mu^{P|N_1} \left( \omega_{1:P}^{(m_1)} \right)$ solve \eqref{eqn:dmvsde_hjb_form3} to obtain control $\zeta(\cdot,\cdot) \left( \omega_{1:P}^{(m_1)} \right)$; \\
    \For{$m_2=1,\ldots,M_2$}{
    Generate realization of random variables $\tilde{\omega}^{(m_2)}$; \\
    Given $\mu^{P|N_1} \left( \omega_{1:P}^{(m_1)} \right)$ and $\zeta(\cdot,\cdot) \left( \omega_{1:P}^{(m_1)} \right)$, solve decoupled MV-SDE with $N_2$ time steps using \eqref{eqn:dmvsde_euler_is_form2};\\
    Compute $G\left(\Bar{X}_\zeta^{P|N_1|N_2}(T) \right) \left( \omega_{1:P}^{(m_1)}, \tilde{\omega}^{(m_2)} \right)$;\\
    Compute $\mathbb{L}^{P|N_1|N_2} \left( \omega_{1:P}^{(m_1)}, \tilde{\omega}^{(m_2)} \right)$ using \eqref{eqn:dlmc_llhood_factor};
    }
    Approximate $\E{G\left(\Bar{X}^{P|N_1|N_2}(T)\right) \mid \mu^{P|N_1} \left( \omega_{1:P}^{(m_1)} \right)}$ by $\frac{1}{M_2} \sum_{m_2=1}^{M_2} G\left(\Bar{X}_\zeta^{P|N_1|N_2}(T)\right) \mathbb{L}^{P|N_1|N_2} \left( \omega_{1:P}^{(m_1)}, \tilde{\omega}^{(m_2)} \right)$; \\
    }
    Approximate $\E{G\left(\Bar{X}^{P|N_1|N_2}(T)\right)}$ by $\frac{1}{M_1} \sum_{m_1=1}^{M_1} \frac{1}{M_2} \sum_{m_2=1}^{M_2} G\left(\Bar{X}_\zeta^{P|N_1|N_2}(T)\right) \mathbb{L}^{P|N_1|N_2} \left( \omega_{1:P}^{(m_1)}, \tilde{\omega}^{(m_2)} \right)$ ;
\end{algorithm}

\subsubsection{Error Analysis}

We bound the global relative error introduced by the DLMC estimator, $\mathcal{A}_{\mathrm{MC}}$, as 
\begin{align}
\label{eqn:dlmc_global_err}
    \frac{\abs{\E{G(X(T))} - \mathcal{A}_{\mathrm{MC}}}}{\abs{\E{G(X(T))}}} & \leq \underbrace{\frac{\abs{\E{G(X(T))} - \E{G\left(\Bar{X}^{P|N_1|N_2}(T)\right)}}}{\abs{\E{G(X(T))}}}}_{=\epsilon_b, \text{ Relative bias}} \nonumber \\
    & \qquad + \underbrace{\frac{\abs{\E{G\left(\Bar{X}^{P|N_1|N_2}(T)\right)} - \mathcal{A}_{\mathrm{MC}}}}{\abs{\E{G(X(T))}}}}_{=\epsilon_s, \text{ Relative statistical error}} \cdot
\end{align}

Although $\mathcal{A}_{\mathrm{MC}}$ should satisfy a given $\tol_\mathrm{r}$ in the sense of \eqref{eqn:mc_objective}, we impose more restrictive conditions, which can be expressed as follows: 
\begin{align}
\label{eqn:dlmc_bias_constraint}
    &\text{Bias Constraint: } \epsilon_b \leq \theta \tol_\mathrm{r} , \\
    \label{eqn:dlmc_stat_constraint}
    &\text{Statistical Constraint: } \prob{\epsilon_s \leq (1-\theta)\tol_\mathrm{r}} > 1 - \alpha ,
\end{align}
for a given tolerance splitting parameter $\theta \in (0,1)$ and confidence level determined by $\alpha$. Let us first analyze the estimator bias, which can be split into two terms:

\begin{align}
    \epsilon_b &= \frac{\abs{\E{G(X(T))} - \E{G\left(\Bar{X}^{P|N_1|N_2}(T)\right)}}}{\abs{\E{G(X(T))}}} \nonumber \\
    \label{eqn:dlmc_bias_split}
    & \leq \underbrace{\frac{\abs{\E{G(X(T))} - \E{G\left(\Bar{X}^{P}(T)\right)}}}{\abs{\E{G(X(T))}}}}_{\text{Relative decoupling error}}
    + \underbrace{\frac{\abs{\E{G\left(\Bar{X}^{P}(T)\right)} - \E{G\left(\Bar{X}^{P|N_1|N_2}(T)\right)}}}{\abs{\E{G(X(T))}}}}_{\text{Relative time discretization error}} , 
\end{align}
and make the following assumptions.
\begin{assump}[Decoupling Error]
	\hspace{1mm} There exists a constant $C_p > 0$ independent of $P$, such that
    \label{ass:decoupling_err}
    \begin{equation*}
        \abs{\E{G(X(T))} - \E{G\left(\Bar{X}^{P}(T)\right)}} \leq C_p P^{-1} \cdot
    \end{equation*}
\end{assump}

\begin{assump}[Time Discretization Error]
	\hspace{1mm} There exist constants $C_{n_1} > 0$ and $C_{n_2} > 0$ independent of $N_1,N_2$, such that
    \label{ass:time_err}
    \begin{equation*}
        \abs{\E{G\left(\Bar{X}^{P}(T)\right)} - \E{G\left(\Bar{X}^{P|N_1|N_2}(T)\right)}} \leq C_{n_1} N_1^{-1} + C_{n_2} N_2^{-1} \cdot
    \end{equation*}
\end{assump}

Assumption \ref{ass:decoupling_err} is motivated by the weak convergence with respect to the number of particles (see \citep{kolokoltsov2019mean}). Assumption \ref{ass:time_err} is motivated by the weak convergence of the Euler--Maruyama scheme for standard SDEs \citep{sde_numerics}. In Section~\ref{sec:results}, we numerically verify these assumptions for the Kuramoto model \eqref{eqn:kuramoto_model}. Note that the constants $C_p,C_{n_1},C_{n_2}$ depend on final time $T$ and the regularity and boundedness of $b(\cdot,\cdot),\sigma(\cdot,\cdot)$ and $G(\cdot)$ and their derivatives. By substituting Assumptions \ref{ass:decoupling_err} and \ref{ass:time_err} into \eqref{eqn:dlmc_bias_split}, the bias bound can be expressed as 
\begin{equation}
\label{eqn:dlmc_bias_split_v2}
    \epsilon_b \abs{\E{G(X(T))}} \leq \frac{C_p}{P} + \frac{C_{n_1}}{N_1} + \frac{C_{n_2}}{N_2} \cdot
\end{equation}

Consider the statistical error constraint \eqref{eqn:dlmc_stat_constraint}; we can approximate the statistical error $\epsilon_s$, by using the Central Limit Theorem as 
\begin{equation}
\label{eqn:dlmc_clt}
    \epsilon_s \approx C_\alpha \sqrt{\Var{\mathcal{A}_{\mathrm{MC}}}} \leq (1-\theta) \tol_\mathrm{r} \abs{\E{G(X(T))}} ,
\end{equation}
where $C_\alpha$ is the $\left(1 - \frac{\alpha}{2}\right)$-quantile for the standard normal distribution. Hence, \eqref{eqn:dlmc_clt} can be expressed as 
\begin{equation}
    \label{eqn:dlmc_variance_constraint}
    \Var{\mathcal{A}_{\mathrm{MC}}} \leq \left(\frac{(1-\theta) \tol_\mathrm{r} \abs{\E{G(X(T))}}}{C_\alpha}\right)^2 
\end{equation}
leading to a constraint on the estimator variance. For notational convenience, let $Y^{P|N_1|N_2} = G\left(\Bar{X}_\zeta^{P|N_1|N_2}(T) \right) \mathbb{L}^{P|N_1|N_2}$ and $Y^{P|N_1|N_2}_{m_1,m_2} \equiv G\left(\Bar{X}_\zeta^{P|N_1|N_2}(T)\right) \mathbb{L}^{P|N_1|N_2} \left( \omega_{1:P}^{(m_1)}, \tilde{\omega}^{(m_2)} \right)$ be random samples of $Y^{P|N_1|N_2}$. Thus, the DLMC estimator variance can be expressed as 




\begin{align}
\label{eqn:dlmc_var_est}
    \Var{\mathcal{A}_{\mathrm{MC}}} &= \Var{\frac{1}{M_1} \sum_{m_1=1}^{M_1} \frac{1}{M_2} \sum_{m_2=1}^{M_2} Y^{P|N_1|N_2}_{m_1,m_2}}    
    = \frac{1}{M_1} \Var{\frac{1}{M_2} \sum_{m_2=1}^{M_2} Y^{P|N_1|N_2}_{1,m_2}} \cdot
\end{align}

Then, by using the law of total variance, we obtain
\begin{align}
    \Var{\mathcal{A}_{\mathrm{MC}}} &= \frac{1}{M_1} \Var{\E{\frac{1}{M_2} \sum_{j=1}^{M_2}Y^{P|N_1|N_2}_{1,j} \mid \mu^{P|N_1} }} + \frac{1}{M_1} \E{\Var{\frac{1}{M_2} \sum_{j=1}^{M_2}Y^{P|N_1|N_2}_{1,j} \mid \mu^{P|N_1} }} \nonumber \\
    \label{eqn:dlmc_total_var}
    &= \frac{1}{M_1} \Var{\E{Y^{P|N_1|N_2} \mid \mu^{P|N_1}}} + \frac{1}{M_1M_2} \E{ \Var{Y^{P|N_1|N_2} \mid \mu^{P|N_1} }},
\end{align}
where $\{Y^{P|N_1|N_2}_{1,j}\}_{j=1}^{M_2}$, represents independent samples conditioned on $\mu^{P|N_1}$. Furthermore, we make the following assumptions.

\begin{assump}
	\hspace{1mm} There exists a constant, $C_1 > 0$, which is independent of $N_1,N_2$, and $P$ such that
    \label{ass:var_1}
    \begin{equation*}
        V^{P|N_1|N_2}_1 = \Var{\E{Y^{P|N_1|N_2} \mid \mu^{P|N_1}}} \leq C_1 P^{-1} \cdot
    \end{equation*}
\end{assump}

\begin{assump}
	\hspace{1mm} There exists a constant, $C_2 > 0$, which is independent of $N_1,N_2$, and $P$ such that
    \label{ass:var_2}
    \begin{equation*}
        V^{P|N_1|N_2}_2 = \E{ \Var{Y^{P|N_1|N_2} \mid \mu^{P|N_1} }} \leq C_2 < \infty \cdot
    \end{equation*}
\end{assump}

Assumption \ref{ass:var_1} is motivated by the convergence of $\mu^{P|N_1}$ to the deterministic mean-field law as $P\longrightarrow\infty$. As the outer variance is related to the randomness in the empirical law $\mu^{P|N_1}$, $V^{P|N_1|N_2}_1$ vanishes with the increase in the number of particles. Assumption~\ref{ass:var_2} is motivated by the bounded conditional variance with respect to randomness driving the dynamics~\eqref{eqn:decoupled_mvsde}. These assumptions are numerically verified for \eqref{eqn:kuramoto_model} in Section~\ref{sec:results}. By substituting Assumptions \ref{ass:var_1} and \ref{ass:var_2} in \eqref{eqn:dlmc_total_var}, constraint \eqref{eqn:dlmc_variance_constraint} can be expressed as
\begin{equation}
\label{eqn:dlmc_var_constraint_v2}
    \frac{C_1}{P M_1} + \frac{C_2}{M_1 M_2} \leq \left(\frac{(1-\theta) \tol_\mathrm{r} \abs{\E{G(X(T))}}}{C_\alpha}\right)^2 \cdot
\end{equation}

\subsubsection{Work Analysis}

This section analyzes the cost to run DLMC Algorithm \ref{alg:dlmc_general}.

\begin{enumerate}
    \item The computational cost required to generate one realization of empirical law using the Euler--Maruyama time-stepping scheme is $N_1 \times P^2$, where $N_1$ denotes the number of time steps.
    \item To numerically solve \eqref{eqn:dmvsde_hjb_form3}, the computational cost is denoted by $\mathcal{W}_{\mathrm{PDE}}$. For standard numerical solvers, $\mathcal{W}_{\mathrm{PDE}} = \order{h^{-d\Gamma}}$, where $\Gamma \geq 1$ is related to the quality of the solver and $h$ is related to the size of mesh grids in time/space.
    \item The computational cost required to generate one realization of decoupled process \eqref{eqn:dmvsde_euler_is_form2} using the Euler-Maruyama time-stepping scheme is $P \times N_2$, where $N_2$ denotes the number of  time steps.
\end{enumerate}

The above-mentioned steps assume a naive method (with computational cost $\order{P}$) to compute the empirical mean for drift and diffusion coefficients in \eqref{eqn:strong_approx_mvsde} and \eqref{eqn:dmvsde_euler_is_form2}. Hence, the total computational cost of Algorithm~\ref{alg:dlmc_general} can be expressed as
\begin{equation}
    \label{eqn:dlmc_cost_v1}
    \mathcal{W} = M_1 \left\{P^2 N_1 + \mathcal{W}_{\mathrm{PDE}} + M_2 \left\{P N_2\right\}\right\} \cdot
\end{equation}

Numerically solving \eqref{eqn:dmvsde_hjb_form3} $M_1$ times is unfeasible. Our revised Algorithm~\ref{alg:dlmc_revised} addresses this issue by solving \eqref{eqn:dmvsde_hjb_form3} offline in a single instance, by using the drift and diffusion coefficients obtained from one realization of the empirical law, using large $\Bar{P},\Bar{N}$ values. This implies that instead of a stochastic control, a deterministic control is obtained that is independent of the different stochastic $P$-particle system realizations in Algorithm~\ref{alg:dlmc_general}. This choice was motivated by the fact that  $\mu^{\Bar{P}|\Bar{N}}$ in the drift and diffusion coefficients converges to the deterministic mean-field law, $\mu_t$, as $\Bar{P},\Bar{N}$ tend to infinity. The online computational cost of running  Algorithm~\ref{alg:dlmc_revised} can be expressed as
\begin{equation}
    \label{eqn:dlmc_cost_v2}
    \mathcal{W} = M_1 \left\{P^2 N_1 + M_2 \left\{P N_2\right\}\right\} , 
\end{equation}
where we exclude the offline cost to numerically solve \eqref{eqn:dmvsde_hjb_form3}. Under the constraints of  \eqref{eqn:dlmc_bias_constraint} and \eqref{eqn:dlmc_var_constraint_v2}, we must determine optimal parameters $P,N_1,N_2,M_1$, and $M_2$ to ensure minimal computational cost \eqref{eqn:dlmc_cost_v2}. This can be presented as the following optimization problem:
\begin{empheq}[left=\empheqlbrace, right =,]{equation}
    \label{eqn:dlmc_min_problem}
    \begin{alignedat}{2}
    \min_{\{P,N_1,N_2,M_1,M_2\}} &\quad \mathcal{W} = M_1 N_1 P^2 + M_1 M_2 N_2 P\\ 
    \text{s.t. } &\frac{C_p}{P} + \frac{C_{n_1}}{N_1} + \frac{C_{n_2}}{N_2} \approx \theta \tol_\mathrm{r} \abs{\E{G(X(T))}} ,\\ 
    & C_{\alpha}^2\left(\frac{C_1}{P M_1} + \frac{C_2}{M_1 M_2}\right) \approx (1-\theta)^2 \tol^2_\mathrm{r} \abs{\E{G(X(T))}}^2 , 
    \end{alignedat}
\end{empheq}
with the solution formulated in Theorem \ref{th:dlmc_optimal_work}.

\begin{algorithm}
    \caption{Revised general DLMC algorithm for decoupled MV-SDE}
\label{alg:dlmc_revised}
    \SetAlgoLined
    \textbf{Offline: } \\
    Generate realisation of law $\mu^{\Bar{P}|\Bar{N}}$ with $\Bar{P}$-particle system and $\Bar{N}$ time steps using \eqref{eqn:dmvsde_discrete_law} with some large $\Bar{P},\Bar{N}$; \\
    Given $\mu^{\Bar{P}|\Bar{N}}$, solve KBE \eqref{eqn:dmvsde_hjb_form3} to obtain control $\zeta(\cdot,\cdot)$; \\
    \textbf{Inputs: } $P,N_1,N_2,M_1,M_2,\zeta(\cdot,\cdot)$; \\
    \For{$m_1=1,\ldots,M_1$}{
    Generate realization of law $\mu^{P|N_1} \left( \omega_{1:P}^{(m_1)} \right)$ with $P$-particle system and $N_1$ time steps using \eqref{eqn:dmvsde_discrete_law}; \\
    \For{$m_2=1,\ldots,M_2$}{
    Generate realization of random variables $\tilde{\omega}^{(m_2)}$; \\
    Given $\mu^{P|N_1} \left( \omega_{1:P}^{(m_1)} \right)$ and $\zeta(\cdot,\cdot)$, solve decoupled MV-SDE with $N_2$ time steps using \eqref{eqn:dmvsde_euler_is_form2};\\
    Compute $G\left(\Bar{X}_\zeta^{P|N_1|N_2}(T)\right) \left( \omega_{1:P}^{(m_1)}, \tilde{\omega}^{(m_2)} \right)$; \\
    Compute $\mathbb{L}^{P|N_1|N_2} \left( \omega_{1:P}^{(m_1)}, \tilde{\omega}^{(m_2)} \right)$ using \eqref{eqn:dlmc_llhood_factor};
    }
    Approximate $\E{G\left(\Bar{X}^{P|N_1|N_2}(T)\right) \mid \mu^{P|N_1} \left( \omega_{1:P}^{(m_1)} \right) }$ by $\frac{1}{M_2} \sum_{m_2=1}^{M_2} G\left(\Bar{X}_\zeta^{P|N_1|N_2}(T)\right) \mathbb{L}^{P|N_1|N_2} \left( \omega_{1:P}^{(m_1)}, \tilde{\omega}^{(m_2)} \right)$; \\
    }
    Approximate $\E{G\left(\Bar{X}^{P|N_1|N_2}(T)\right)}$ by $\frac{1}{M_1} \sum_{m_1=1}^{M_1} \frac{1}{M_2} \sum_{m_2=1}^{M_2} G\left(\Bar{X}_\zeta^{P|N_1|N_2}(T)\right) \mathbb{L}^{P|N_1|N_2} \left( \omega_{1:P}^{(m_1)}, \tilde{\omega}^{(m_2)} \right)$ ;
\end{algorithm}

\begin{thm}[Optimal DLMC Complexity]
\label{th:dlmc_optimal_work}
\hspace{1mm} Consider the DLMC estimator in \eqref{eqn:dmvsde_dlmc_est}, obtained from Algorithm~\ref{alg:dlmc_revised}. For any $\tol_\mathrm{r} > 0$, there exist optimal parameters $\{P,N_1,N_2,M_1,M_2\}$ such that \eqref{eqn:dlmc_bias_constraint} and \eqref{eqn:dlmc_var_constraint_v2} hold. The optimal computational work is given as
    \begin{equation}
    \label{eqn:dlmc_optimal_work}
        \mathcal{W} = M_1 N_1 P^2 + M_1 M_2 N_2 P = \order{\tol_\mathrm{r}^{-4}} \cdot
    \end{equation}
\end{thm}

\begin{proof}
   Refer to Appendix \ref{appendix:a3}.
\end{proof}

\begin{rem}[Optimal MC Complexity for Particle Systems]
\hspace{1mm} The introduction of DLMC for this problem crucially guarantees $\order{\tol_\mathrm{r}^{-4}}$ complexity for the proposed estimator. From \eqref{eqn:dlmc_min_problem} and Theorem \ref{th:dlmc_optimal_work}, the use of a single realization of the particle system ($M_1=1$) results in an increase in complexity to $\order{\tol_\mathrm{r}^{-5}}$. Moreover, in comparison to previous MC estimators~\citep{mlmc_mvsde}, the proposed IS scheme for the DLMC estimator also considerably reduces the associated constant for rare-event probabilities. This is validated in Section \ref{sec:results}.
\end{rem}




\subsubsection{Adaptive DLMC Algorithm}

We formulate a DLMC algorithm that adaptively selects optimal parameters $P,N_1,N_2, M_1$, and $M_2$. For computational convenience, we restrict our choices for $P,N_1$, and $N_2$, by implementing the following hierarchies:
\begin{itemize}
    \item $P_\ell = P_0 \cross \tau^\ell \quad,\ell=0,\ldots,L$
    \item $(N_1)_\ell = (N_2)_\ell = N_\ell = N_0 \cross \tau^{\ell} \quad,\ell=0,\ldots,L$
\end{itemize}

Herein, we choose $\tau=2$. The choice of $N_1=N_2$ is nearly optimal as the optimal values of $N_1$ and $N_2$ are both $\order{\tol_\mathrm{r}^{-1}}$ (Appendix~\ref{appendix:a3}). To build an adaptive algorithm, we must estimate the bias and variances $V_{1,\ell}$ and $V_{2,\ell}$ for some level $\ell$ in a cheap and robust manner. To satisfy the bias constraint, we must select a level $\ell = L$ that satisfies \eqref{eqn:dlmc_bias_constraint}. For convenience of the following analysis, we denote $G = G(X(T))$ and its discretization at level $\ell$ as $G_\ell = G\left(\Bar{X}_\zeta^{P_\ell|N_\ell|N_\ell}(T)\right)$. The likelihood factor, $\mathbb{L}_\ell = \mathbb{L}^{P_\ell|N_\ell|N_\ell}$, at level $\ell$ is computed using \eqref{eqn:dlmc_llhood_factor}. Then, the relative bias for level $\ell$ can be expressed as 

\begin{align}
\label{adlmc_bias_level}
    \epsilon_b &= \frac{\abs{\E{G - G_\ell}}}{\abs{\E{G}}} \leq \frac{1}{\abs{\E{G}}} \left( \frac{C_p}{P_\ell} + \frac{C_n}{N_\ell} \right) \cdot
\end{align}

We use the Richardson extrapolation technique~\citep{mlmc_richardson_extra} to estimate bias for $\tau=2$ at level $\ell$ as

\begin{equation}
\label{eqn:adlmc_bias_richardson}
    \abs{\E{G - G_\ell}} \approx 2 \abs{\E{G_{\ell+1} - G_\ell}} \cdot
\end{equation}

We use a robust DLMC estimate of $\E{\Delta G_{\ell+1}} = \E{G_{\ell+1} - G_\ell}$ by using $\bar{M}_1,\bar{M}_2$ samples according to  Algorithm~\ref{alg:adlmc_est_bias}, which incorporates an antithetic sampler~\citep{mlmc_mvsde} to estimate the bias using sufficiently correlated samples of $G_{\ell+1}$ and $G_\ell$. More details regarding the antithetic sampler can be found in Algorithm \ref{alg:adlmc_est_bias}.

\begin{algorithm}
    \caption{Estimating $\E{\Delta G_\ell}$ with an antithetic sampler}
\label{alg:adlmc_est_bias}
    \SetAlgoLined
    \textbf{Inputs: } $\ell,M_1,M_2,\zeta(\cdot,\cdot)$; \\
    \For{$m_1=1,\ldots,M_1$}{
    Generate realization of random variables $\omega_{1:P_\ell}^{(m_1)}$; \\
    Generate $\mu^{P_{\ell}|N_{\ell}} \left( \omega_{1:P_\ell}^{(m_1)} \right)$ realization with $P_{\ell}$-particle system and $N_{\ell}$ time steps using \eqref{eqn:dmvsde_discrete_law}; \\
    From $\mu^{P_{\ell}|N_{\ell}} \left( \omega_{1:P_\ell}^{(m_1)} \right)$, generate $\mu^{P_{\ell}-1|N_{\ell}-1} \left( \omega_{1:\frac{P_\ell}{2}}^{(m_1)} \right)$ and $\mu^{P_{\ell}-1|N_{\ell}-1} \left( \omega_{\frac{P_\ell}{2}:P_\ell}^{(m_1)} \right)$; \\
   \For{$m_2=1,\ldots,M_2$}{
   Generate realization of random variables $\tilde{\omega}^{(m_2)}$; \\
    Given $\mu^{P_{\ell}|N_{\ell}} \left( \omega_{1:P_\ell}^{(m_1)} \right)$ and $\zeta(\cdot,\cdot)$, solve decoupled MV-SDE with $N_{\ell}$ time steps using \eqref{eqn:dmvsde_euler_is_form2};\\
    Compute $G_{\ell} \mathbb{L}_{\ell} \left( \omega_{1:P_\ell}^{(m_1)}, \tilde{\omega}^{(m_2)} \right)$; \\
    Given $\mu^{P_{\ell}-1|N_{\ell}-1} \left( \omega_{1:\frac{P_\ell}{2}}^{(m_1)} \right)$ and $\zeta(\cdot,\cdot)$, solve decoupled MV-SDE with $N_{\ell}-1$ time steps using \eqref{eqn:dmvsde_euler_is_form2};\\
    Compute $G_{\ell} \mathbb{L}_{\ell} \left( \omega_{1:\frac{P_\ell}{2}}^{(m_1)}, \tilde{\omega}^{(m_2)} \right)$; \\
    Given $\mu^{P_{\ell}-1|N_{\ell}-1} \left( \omega_{\frac{P_\ell}{2}:P_\ell}^{(m_1)} \right)$ and $\zeta(\cdot,\cdot)$, solve decoupled MV-SDE with $N_{\ell}-1$ time steps using \eqref{eqn:dmvsde_euler_is_form2};\\
    Compute $G_{\ell} \mathbb{L}_{\ell} \left( \omega_{\frac{P_\ell}{2}:P_\ell}^{(m_1)}, \tilde{\omega}^{(m_2)} \right)$; \\
    }
    }
    Approximate $\E{G_{\ell} -G_{\ell-1}}$ by $\frac{1}{M_1} \sum_{m_1=1}^{M_1} \frac{1}{M_2} \sum_{m_2=1}^{M_2} \left(G_{\ell} \mathbb{L}_{\ell}\right)^{(m_1,m_2)} - \frac{G_{\ell} \mathbb{L}_{\ell} \left( \omega_{1:\frac{P_\ell}{2}}^{(m_1)}, \tilde{\omega}^{(m_2)} \right) +G_{\ell} \mathbb{L}_{\ell} \left( \omega_{\frac{P_\ell}{2}:P_\ell}^{(m_1)}, \tilde{\omega}^{(m_2)} \right)}{2}$; \\
\end{algorithm}

Given that some level $\ell = L$ satisfies \eqref{eqn:dlmc_bias_constraint}, we must find optimal parameters $M_1$ and $M_2$ that satisfy \eqref{eqn:dlmc_variance_constraint}. This can be presented as the following  optimization problem:
\begin{empheq}[left=\empheqlbrace, right =,]{equation}
    \label{eqn:adlmc_min_problem}
    \begin{alignedat}{2}
    \min_{\{M_1,M_2\}} \quad & \mathcal{W} = M_1 N_L P_L^2 + M_1 M_2 N_L P_L\\ 
    & C_{\alpha}^2\left(\frac{V_{1,L}}{M_1} + \frac{V_{2,L}}{M_1 M_2}\right) \approx(1-\theta)^2 \tol_\mathrm{r}^2 \abs{\E{G(X(T))}}^2  , 
    \end{alignedat}
\end{empheq} 
where $V_{1,L} = V_1^{P_L|N_L|N_L}$ and $V_{2,L} = V_2^{P_L|N_L|N_L}$. By  solving \eqref{eqn:adlmc_min_problem}, we get
\begin{align}
    \label{eqn:adlmc_optimal_params}
    M_1 &= \left(V_{1,L} + \sqrt{\frac{V_{1,L} V_{2,L}}{P_L}}\right) \frac{C_\alpha^2}{(1-\theta)^2 \tol_\mathrm{r}^2 \abs{\E{G(X(T))}}^2}, \quad
    M_2 = \sqrt{\frac{V_{2,L} P_L}{V_{1,L}}} \cdot
\end{align}

In principle, $M_1$ and $M_2$ obtained in \eqref{eqn:adlmc_optimal_params} are real-valued. In practice, however, we use the next highest integer values for $M_1$ and $M_2$. From \eqref{eqn:adlmc_optimal_params}, one observes that we must estimate $V_{1,L}$ and $V_{2,L}$ to obtain a DLMC estimator that satisfies \eqref{eqn:dlmc_variance_constraint}. The proposed adaptive algorithm uses a heuristic DLMC estimator for these variances, as stated in Algorithm~\ref{alg:adlmc_est_constants} in Appendix~\ref{app:est_constants}, with $\bar{M}_1$ and $\bar{M}_2$ samples.

The proposed adaptive algorithm updates a heuristic estimate for $\E{G(X(T))}$ (the quantity of interest) to check the estimator's relative error at each step. To initialize the adaptive algorithm, we produce an initial rough DLMC estimate by using IS with $\tilde{M}_1$ and $\tilde{M}_2$ samples. With this, we have all the components required to formulate an adaptive DLMC algorithm, as shown in Algorithm~\ref{alg:adlmc}. 



\begin{algorithm}
    \caption{Adaptive DLMC algorithm for decoupled MV-SDE}
\label{alg:adlmc}
    \SetAlgoLined
    \textbf{Offline: } \\
    Generate $\mu^{\Bar{P}|\Bar{N}}$ realisation with $\Bar{P}$-particle system and $\Bar{N}$ time steps \eqref{eqn:dmvsde_discrete_law} with arbitrarily large $\Bar{P},\Bar{N}$; \\
    Given $\mu^{\Bar{P}|\Bar{N}}$, solve KBE \eqref{eqn:dmvsde_hjb_form3} to obtain control $\zeta(\cdot,\cdot)$; \\
    \textbf{Input: } $P_0,N_0,\tol_{\mathrm{r}},\zeta(\cdot,\cdot)$; \\
     $\ell=0$; \\
     Estimate $\hat{\alpha} = \E{G_0}$ using $P_0,N_0,\Tilde{M}_1,\Tilde{M}_2,\zeta(\cdot,\cdot)$ in \textbf{Algorithm \ref{alg:dlmc_revised}};\\
    \While{$\mathrm{Bias} > \theta \tol_{\mathrm{r}} \hat{\alpha}$}{
    $P_\ell=P_0\cross2^\ell,\quad N_\ell = N_0 \cross 2^{\ell}$; \\
    Estimate $V_{1,\ell}$ and $V_{2,\ell}$ using $P_\ell,N_\ell, \Bar{M}_1, \Bar{M}_2, \zeta(\cdot,\cdot)$ in \textbf{Algorithm \ref{alg:adlmc_est_constants}}; \\
    Compute optimal $M_1,M_2$ with estimated $V_{1,\ell},V_{2,\ell}$ using \eqref{eqn:adlmc_optimal_params}; \\
    Estimate Bias = $\frac{\abs{2 \E{\Delta G_{\ell+1}}}}{\hat{\alpha}}$ with $P_\ell,N_\ell,\hat{M}_1,\hat{M}_2,\zeta(\cdot,\cdot)$ in \textbf{Algorithm \ref{alg:adlmc_est_bias}}; \\
    Update $\hat{\alpha} = \E{G_\ell}$ using $P_\ell,N_\ell,M_1,M_2,\zeta(\cdot,\cdot)$ in \textbf{Algorithm \ref{alg:dlmc_revised}}; \\
    $\ell \longleftarrow \ell+1$;
    }
    $\mathcal{A}_{\mathrm{MC}} = \hat{\alpha}$.
\end{algorithm}

\section{Numerical Results}
\label{sec:results}

This section provides numerical evidence for the assumptions and the computational complexity derived in Section~\ref{sec:dlmc}. The reported results focus on the Kuramoto model~\eqref{eqn:kuramoto_model} with $\sigma=0.4$, $T=1$, $(x_0)_p \sim \mathcal{N}(0,0.2)$, and $\nu_p \sim \mathcal{U}(-0.2,0.2)$ for all $p = 1,\ldots,P$. We implemented our DLMC algorithm on both nonrare- and rare-event observables. We demonstrate the effectiveness of our method over naive MC. 

\subsection{Objective function $G(x) = \cos{x}$}

We implemented the proposed DLMC algorithm for the nonrare observable $G(x) = \cos{x}$. Due to the fact that this is not a rare-event observable, IS is not required in the algorithms, i.e., $\zeta(t,x) = 0,\quad \forall (t,x) \in [0,T] \cross \mathbb{R}^d$. First, we verify Assumption~\ref{ass:decoupling_err}. As $\E{G(X(T))}$ is unknown, we used Richardson extrapolation to obtain the following estimate, which can be numerically computed.



\begin{equation}
    \label{eqn:dlmc_cp_est}
    \abs{\E{G(\Bar{X}^{2P}(T))} - \E{G(\Bar{X}^P(T))}} \leq \frac{C_p}{2P} \cdot
\end{equation}

To achieve a robust numerical estimate, we coupled the expectation computations in \eqref{eqn:dlmc_cp_est} by using the antithetic sampler (Algorithm~\ref{alg:adlmc_est_bias}). Assumption~\ref{ass:time_err} was verified using the following three estimates for the time-discretization error based on parameters $N_1$ and $N_2$.
\begin{align}
\label{eqn:dlmc_dt1_error}
    \abs{\E{G(\Bar{X}^{P|2N_1|N_2}(T))} - \E{G(\Bar{X}^{P|N_1|N_2}(T))}} &\leq \frac{C_{n_1}}{2 N_1} , \\
    \label{eqn:dlmc_dt2_error}
    \abs{\E{G(\Bar{X}^{P|N_1|2N_2}(T))} - \E{G(\Bar{X}^{P|N_1|N_2}(T))}} &\leq \frac{C_{n_2}}{2 N_2} , \\
    \label{eqn:dlmc_dt_comb_error}
    \abs{\E{G(\Bar{X}^{P|2N|2N}(T))} - \E{G(\Bar{X}^{P|N|N}(T))}} &\leq \frac{C_n}{2N} \cdot
\end{align}

Figure~\ref{fig:dlmc_error} verifies the proposed orders of bias convergence with respect to $P,N_1$, and $N_2$. 
Figure~\ref{fig:dlmc_variance} displays the estimated $V^{P|N_1|N_2}_1$ and $V^{P|N_1|N_2}_2$ using Algorithm~\ref{alg:adlmc_est_constants}, where $V^{P|N_1|N_2}_1$ converges with $\order{P^{-1}}$ and $V^{P|N_1|N_2}_2$ is nearly  constant. Figures \ref{fig:dlmc_error} and \ref{fig:dlmc_variance} validate  the assumptions made in Section~\ref{sec:dlmc}. 

\begin{figure}
     \centering
     \begin{subfigure}[b]{.45\textwidth}
         \centering
         \includegraphics[width=\textwidth]{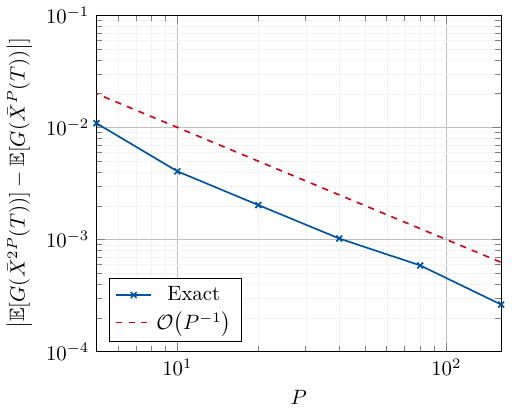}
         \caption{Verifying \eqref{eqn:dlmc_cp_est}. DLMC estimate for $\abs{\E{G(\Bar{X}^{2P}(T))} - \E{G(\Bar{X}^P(T))}}$ by using Algorithm~\ref{alg:dlmc_revised} with inputs $N_1=N_2=128,M_1=100$ and $M_2=10^3$ with respect to number of particles $P$.}
         \label{fig:dlmc_dc_error}
     \end{subfigure}
     \hfill
     \begin{subfigure}[b]{.45\textwidth}
         \centering
         \includegraphics[width=\textwidth]{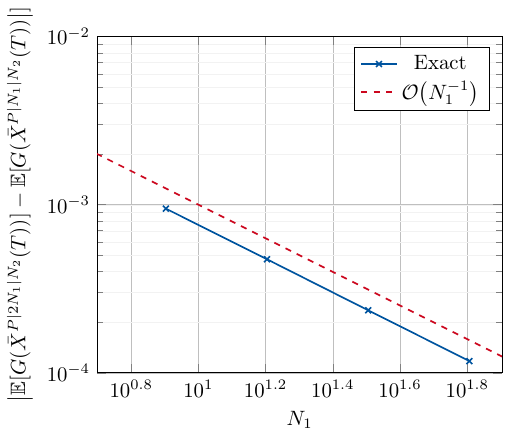}
         \caption{Verifying \eqref{eqn:dlmc_dt1_error}: DLMC estimator for $|\E{G(\Bar{X}^{P|2N_1|N_2}(T))} -$ $ \E{G(\Bar{X}^{P|N_1|N_2}(T))}|$ using Algorithm~\ref{alg:dlmc_revised} with inputs $P=80,N_2=256,M_1=10^2$ and $M_2=10^3$ with respect to the number of time steps, $N_1$}
         \label{fig:dlmc_dt1_error}
     \end{subfigure}
     \hfill
     \begin{subfigure}[b]{.45\textwidth}
         \centering
         \includegraphics[width=\textwidth]{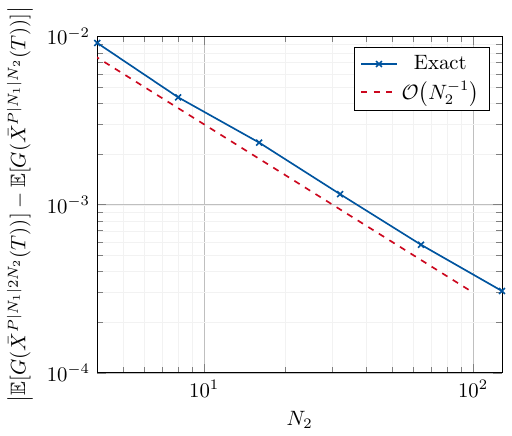}
         \caption{Verifying \eqref{eqn:dlmc_dt2_error}: DLMC estimator for $|\E{G(\Bar{X}^{P|N_1|2N_2}(T))} -$  $\E{G(\Bar{X}^{P|N_1|N_2}(T))}|$ using Algorithm~\ref{alg:dlmc_revised} with inputs $P=80,N_1=256,M_1=10^2$, and $M_2=10^3$ with respect to the number of time steps, $N_2$}
         \label{fig:dlmc_dt2_error}
     \end{subfigure}
     \hfill
     \begin{subfigure}[b]{.45\textwidth}
         \centering
         \includegraphics[width=\textwidth]{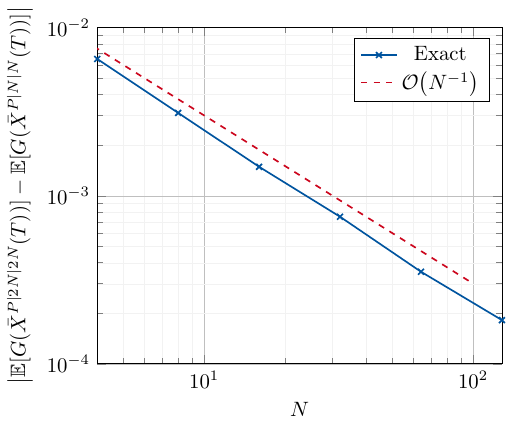}
         \caption{Verifying \eqref{eqn:dlmc_dt_comb_error}: DLMC estimator for $|\E{G(\Bar{X}^{P|2N|2N}(T))} - $ $ \E{G(\Bar{X}^{P|N|N}(T))}|$ using Algorithm~\ref{alg:dlmc_revised} with inputs $P=80,M_1=10^2$, and $M_2=10^3$ with respect to the number of time steps, $N$}
         \label{fig:dlmc_dt_comb_error}
     \end{subfigure}
        \caption{Verifying Assumptions~\ref{ass:decoupling_err} and \ref{ass:time_err} for Kuramoto model \eqref{eqn:kuramoto_model} for $G(x)=\cos{x}$. }
        \label{fig:dlmc_error}
\end{figure}

\begin{figure}
     \centering
     \begin{subfigure}[b]{0.45\textwidth}
         \centering
         \includegraphics[width=\textwidth]{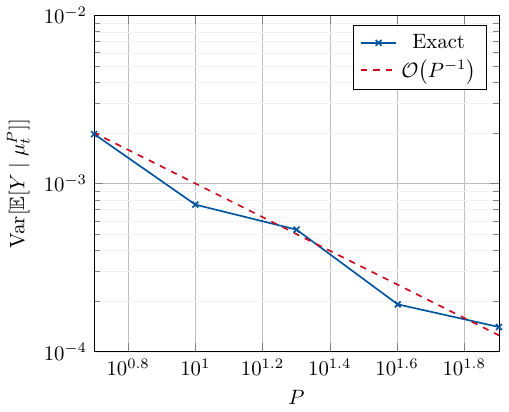}
         \caption{Verifying Assumption~\ref{ass:var_1}: DLMC estimator for $V^{P|N_1|N_2}_1$ by using Algorithm~\ref{alg:adlmc_est_constants} with inputs $N_1=N_2=128,M_1=10^2$, and $M_2=10^4$ with respect to number of particles $P$}
         \label{fig:dlmc_vare}
     \end{subfigure}
     \hfill
     \begin{subfigure}[b]{0.45\textwidth}
         \centering
         \includegraphics[width=\textwidth]{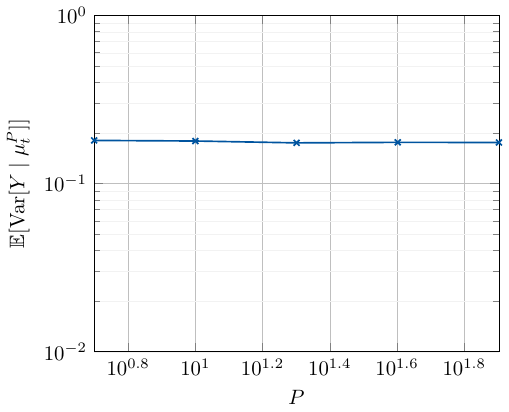}
         \caption{Verifying Assumption~\ref{ass:var_2}: DLMC estimator for $V^{P|N_1|N_2}_2$ using Algorithm~\ref{alg:adlmc_est_constants} with inputs $N_1=N_2=256,M_1=10^2$, and $M_2=10^4$ with respect to number of particles $P$}
         \label{fig:dlmc_evar}
     \end{subfigure}
        \caption{Verifying Assumptions~\ref{ass:var_1} and \ref{ass:var_2} for Kuramoto model \eqref{eqn:kuramoto_model}.}
        \label{fig:dlmc_variance}
\end{figure}



We tested the adaptive DLMC algorithm (Algorithm~\ref{alg:adlmc}) on the Kuramoto model \eqref{eqn:kuramoto_model}, utilizing $P_0=5,N_0=4$ as inputs. We used Algorithm~\ref{alg:adlmc_est_constants} with $\Bar{M}_1 = 100$ and $\Bar{M}_2 = 3000$ samples to estimate $V_{1,\ell}$ and $V_{2,\ell}$ for each level $\ell$. In addition, we used Algorithm~\ref{alg:adlmc_est_bias} with $\hat{M}_1 = M_1$ and $\hat{M}_2 = M_2$ to estimate the bias, where $M_1$ and $M_2$ are the optimal number of samples obtained from \eqref{eqn:adlmc_optimal_params}. Figure \ref{fig:adlmc_runtime} shows the  computational runtime for Algorithm~\ref{alg:adlmc} for different error tolerances. The runtimes for sufficiently small tolerances follow the predicted theoretical rate, $\order{\tol^{-4}}$, derived in Theorem~\ref{th:dlmc_optimal_work}. Figure \ref{fig:adlmc_error} shows the exact DLMC estimator error for separate runs of Algorithm~\ref{alg:adlmc} for different prescribed absolute error tolerances ($\tol$), where the exact error was computed using a reference DLMC approximation with $\tol=10^{-3.5}$. Figure \ref{fig:adlmc_error} clearly shows that Algorithm~\ref{alg:adlmc} produces an estimate that satisfies error constraint~\eqref{eqn:mc_objective}. 

\begin{figure}
     \centering
     \begin{subfigure}[b]{0.45\textwidth}
         \centering
         \includegraphics[width=\textwidth]{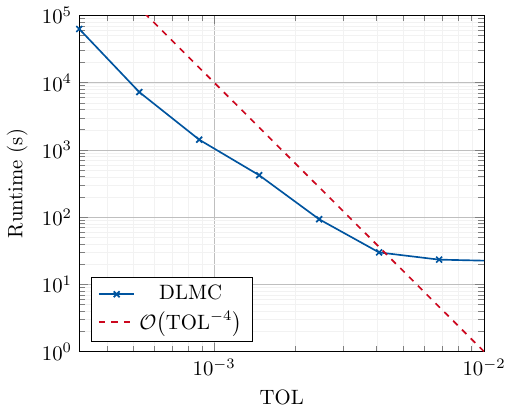}
         \caption{Runtime with respect to $\tol$}
         \label{fig:adlmc_runtime}
     \end{subfigure}
     \hfill
     \begin{subfigure}[b]{0.45\textwidth}
         \centering
         \includegraphics[width=\textwidth]{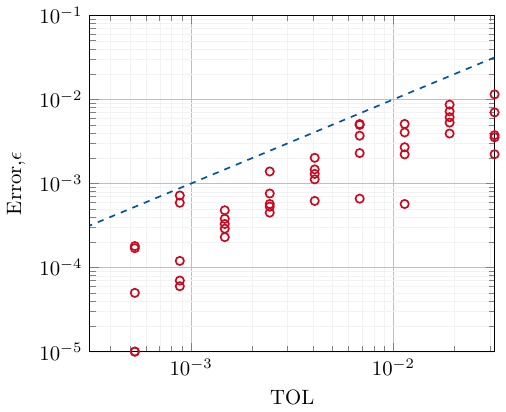}
         \caption{Global error with respect to $\tol$}
         \label{fig:adlmc_error}
     \end{subfigure}
        \caption{Adaptive DLMC Algorithm~\ref{alg:adlmc} applied to Kuramoto example \eqref{eqn:kuramoto_model} for $G(x) = \cos{x}$.}
        \label{fig:adlmc}
\end{figure}

\subsection{Objective function $G(x) = \mathbbm{1}_{\{x>K\}}$}

For a sufficiently large value of $K$, $G(x) = \mathbbm{1}_{\{x>K\}}$ corresponds to the probability of a rare event. Figures~\ref{fig:test1} and \ref{fig:adlmc_rare} use a threshold of $K=2$, corresponding to a probability of approximately $2.53 \times 10^{-4}$. We implemented the IS scheme (Section~\ref{sec:dlmc}) with IS control $\zeta(\cdot,\cdot)$ for the one-dimensional Kuramoto model \eqref{eqn:kuramoto_model} obtained by numerically  solving \eqref{eqn:dmvsde_hjb_form3} using finite differences and linear interpolation throughout the domain. Two numerical experiments were initially conducted, verifying variance reduction from IS. In the first experiment shown in Figure \ref{fig:test1_mcis}, we verified variance reduction on the MC estimator of the inner expectation conditioned on an empirical law, $\mu^{P|N_1}$. To obtain Figure \ref{fig:test1_mcis}, we acquired $\mu^{P|N_1}$ empirically by using the stochastic $P$-particle system with $P=200$ and $N_1=32$. We used this law to obtain both the IS control $\zeta(\cdot,\cdot)$ as well as an input to all realizations of the decoupled MV-SDE \eqref{eqn:decoupled_mvsde}. We simulated the decoupled MV-SDE by using $N_2=32$ time steps. Figure \ref{fig:test1_mcis} presents a comparison of the squared coefficients of variation for the MC estimator of the inner conditional expectation with and without IS with respect to the number of sample paths $M$ for the decoupled MV-SDE. The plots verify that the squared coefficient of variation for the estimator reduces approximately $6000$-fold with IS. In the second experiment, we used $\Bar{P}=200$ particles and $\Bar{N}=100$ time steps in the stochastic particle system to estimate empirical $\mu^{\Bar{P}|\Bar{N}}$, and subsequently obtained the optimal IS control. Then, we set $P=100,N_1=N_2=32,M_1=10^3$, and varied $M_2$ as inputs to Algorithm~\ref{alg:dlmc_revised}. Figure~\ref{fig:test2_mcis} shows the squared coefficient of variation of the DLMC estimator with respect to $M_2$. We observed that the IS estimator displays a remarkably reduced variance (approximately $1000$-fold).

\begin{figure}
     \centering
     \begin{subfigure}[b]{0.45\textwidth}
         \centering
         \includegraphics[width=\textwidth]{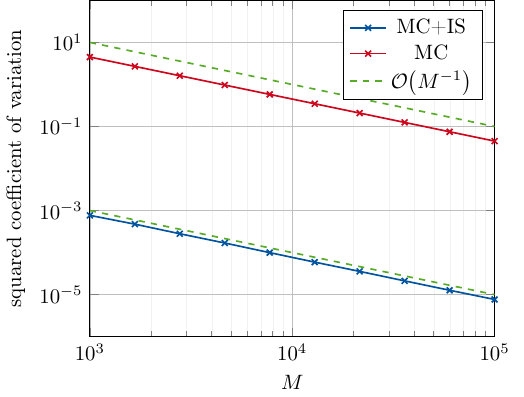}
         \caption{MC estimator squared coefficient of variation for $\prob{\Bar{X}(T)>K}$, conditional on fixed empirical $\mu^{P|N_1}$ with respect to number of sample paths, $M$}
         \label{fig:test1_mcis}
     \end{subfigure}
     \hfill
     \begin{subfigure}[b]{0.45\textwidth}
         \centering
         \includegraphics[width=\textwidth]{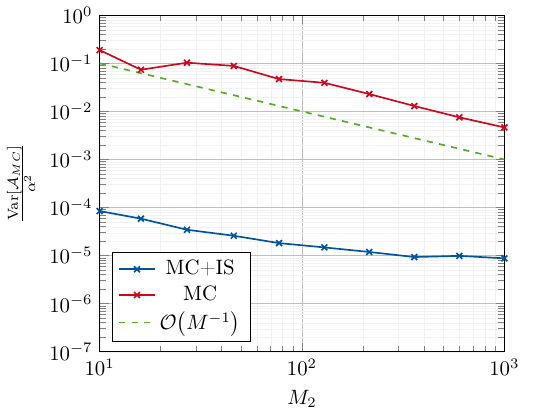}
         \caption{DLMC estimator squared coefficient of variation for $\prob{\Bar{X}(T)>K}$ with respect to the number of sample paths in the inner loop, $M_2$}
         \label{fig:test2_mcis}
     \end{subfigure}
     \caption{Variance reduction of the DLMC estimator using IS on Kuramoto example \eqref{eqn:kuramoto_model} for $G(x) = \mathbbm{1}_{\{x>K\}}$.}
        \label{fig:test1}
\end{figure}


We applied Algorithm~\ref{alg:adlmc} to Kuramoto model \eqref{eqn:kuramoto_model} with $\Bar{P}=1000$ particles and $\Bar{N}_1=100$ time steps to estimate empirical $\mu^{\Bar{P}|\Bar{N}}$ by using the stochastic particle system and obtaining control $\zeta(\cdot,\cdot)$. We used $P_0=5,N_0=4$, and a relative error tolerance $\tol_{\mathrm{r}}$ as inputs to Algorithm \ref{alg:adlmc}. The following heuristics were employed to ensure the robustness of the proposed algorithm.
\begin{itemize}
    \item We used $\Tilde{M}_1=10^3$ and $\Tilde{M}_2=10^2$ to obtain an initial rough estimate for the required quantity to aid in quantifying the required tolerance.
    \item Algorithm~\ref{alg:adlmc_est_constants} was employed to estimate $V_{1,\ell}$ and $V_{2,\ell}$ with $\Bar{M_1}=50$ and $\Bar{M}_2 = 10^3$ for the first three levels, i.e., $\ell=1,2,3$. In addition, Assumptions~\ref{ass:var_1} and \ref{ass:var_2} were employed for the  subsequent levels to linearly extrapolate $V_{1,\ell}$ and $V_{2,\ell}$. 
    \item Algorithm~\ref{alg:adlmc_est_bias} was used to estimate bias, with $\hat{M}_1=\max(M_1,100)$ and $\hat{M}_2=\max(M_2,50)$, where $M_1,M_2$ are  the computed optimal sample sizes. The estimated bias was compared with the  extrapolated bias from the last two levels, and the maximum of the three values was selected, ensuring the bias estimate robustness for $\ell>3$.
\end{itemize}

Figure \ref{fig:mcis_qoi} illustrates that $K=2$ corresponds to a probability of approximately $2.53 \times 10^{-4}$. Figure \ref{fig:mcis_global_err} shows the exact relative error of our DLMC estimator for different runs of Algorithm \ref{alg:adlmc} over various prescribed relative error tolerances. We used a reference DLMC approximation computed with $\tol_{\mathrm{r}} = 1.5\%$. Figure \ref{fig:mcis_runtime} shows that the computational runtime closely follows the predicted theoretical rate of $\order{\tol_{\mathrm{r}}^{-4}}$ for small relative tolerances. Additionally, we compared the estimated computational work, given by \eqref{eqn:dlmc_cost_v2}, for the IS and crude DLMC methods. Because running a crude DLMC is infeasible for rare events, we used a heuristic estimate of the computational cost of crude DLMC without actually running the algorithm. Figure \ref{fig:mcis_work} provides numerical evidence that the IS estimator reduced the computational cost to achieve a prescribed relative error tolerance, by multiple orders (three-orders of magnitude in this case). This implies that our IS estimator dramatically reduces the constant associated with estimating rare-event probabilities.

\begin{figure}
     \centering
     \begin{subfigure}[b]{0.45\textwidth}
         \centering
         \includegraphics[width=\textwidth]{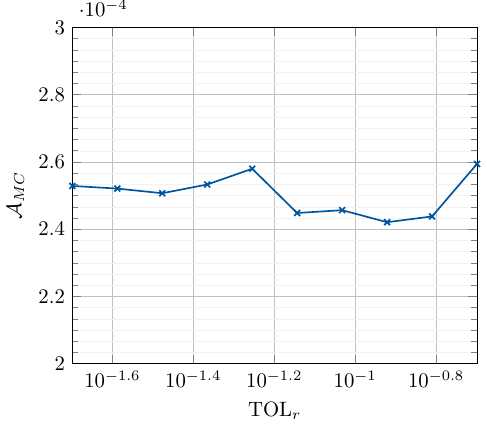}
         \caption{Estimator $\mathcal{A}_{\mathrm{MC}}$ for quantity of interest}
         \label{fig:mcis_qoi}
     \end{subfigure}
     \hfill
     \begin{subfigure}[b]{0.45\textwidth}
         \centering
         \includegraphics[width=\textwidth]{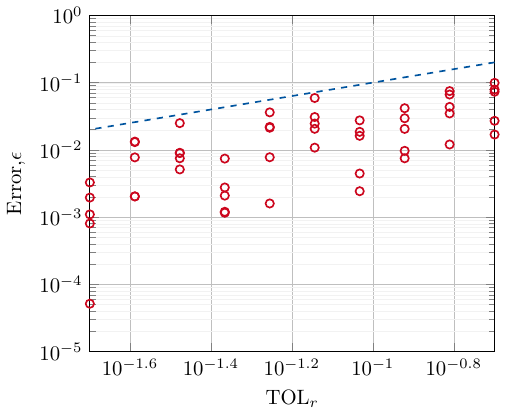}
         \caption{Relative estimator error with respect to relative tolerance, $\tol_{\mathrm{r}}$}
         \label{fig:mcis_global_err}
     \end{subfigure}
     \hfill
     \begin{subfigure}[b]{0.45\textwidth}
         \centering
         \includegraphics[width=\textwidth]{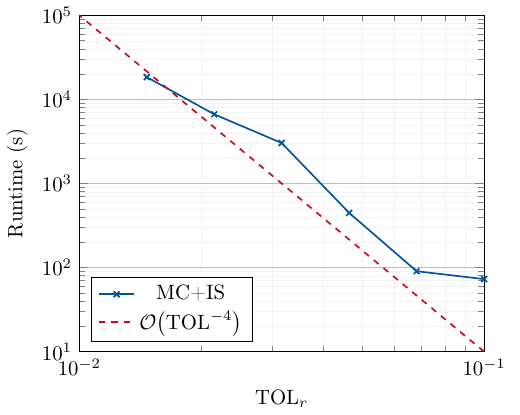}
         \caption{Computational runtime with respect to relative tolerance, $\tol_{\mathrm{r}}$}
         \label{fig:mcis_runtime}
     \end{subfigure}
     \hfill
     \begin{subfigure}[b]{0.45\textwidth}
         \centering
         \includegraphics[width=\textwidth]{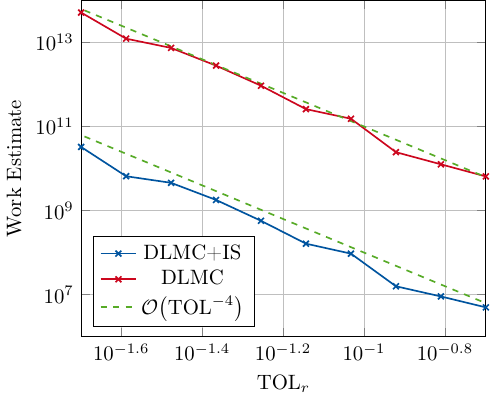}
         \caption{Computational work estimate with respect to relative tolerance, $\tol_{\mathrm{r}}$}
         \label{fig:mcis_work}
     \end{subfigure}
     \caption{Algorithm~\ref{alg:adlmc} applied to Kuramoto example \eqref{eqn:kuramoto_model} for $G(x) = \mathbbm{1}_{\{x>K\}}$.}
        \label{fig:adlmc_rare}
\end{figure}

Table \ref{tab:adlmc_comparison} shows the number of samples required to reach a given relative tolerance with and without IS for different thresholds ($K$). For DLMC without IS, the number of samples required to satisfy a given relative tolerance increases as the probability reduces~\citep{is_general_ref}. Our IS scheme reduced both $M_1$ and $M_2$. In fact, the required number of samples for IS is considerably lesser than that without IS, and it remains of the same order regardless of the event rarity. Thus, the proposed DLMC estimator with IS numerically achieves the bounded relative error property~\citep{is_relative_error}.  

\begin{table}
    \centering
    \begin{tabular}{||c|c|c|c|c|c|c|c||}
    \hline
        \multicolumn{2}{||c|}{} & \multicolumn{2}{|c|}{} & \multicolumn{2}{|c|}{DLMC+IS} & \multicolumn{2}{|c||}{DLMC} \\
        \hline
        K & $\mathcal{A}_{\mathrm{MC}}$ & $\tol_{\text{r}}$ & $\ell$ & $M_1$ & $M_2$ & $M_1$ & $M_2$  \\
        \hline
        \multirow{3}{*}{1} & \multirow{3}{*}{$5.6 \times 10^{-2}$} & 20\% & 3 & 59 & 12 & 157 & 65 \\
        & & 10\% & 4 & 111 & 27 & 284 & 131 \\
        & & 5\% & 5 & 167 & 58 & 456 & 297 \\
        \hline
        \multirow{3}{*}{1.5} & \multirow{3}{*}{$6.8 \times 10^{-3}$} & 20\% & 4 & 41 & 24 & 230 & 325 \\
        & & 10\% & 5 & 64 & 48 & 386 & 700 \\
        & & 5\% & 6 & 105 & 102 & 662 & 1491 \\
        \hline
        \multirow{3}{*}{2} & \multirow{3}{*}{$2.53 \times 10^{-4}$} & 20\% & 5 & 17 & 48 & 477 & 3629 \\
        & & 10\% & 6 & 33 & 90 & 913 & 6929 \\
        & & 5\% & 7 & 57 & 186 & 1711 & 14864 \\
        \hline
    \end{tabular}
    \caption{Naive DLMC and DLMC+IS methods estimating rare event probabilities to satisfy $\tol_{\mathrm{r}}$ for different thresholds ($K$).} 
    \label{tab:adlmc_comparison}
\end{table}

\section{Conclusion}
\label{sec:conclusion}

This work has shown both theoretically and numerically, under certain assumptions which could be verified numerically, the effectiveness of the novel DLMC estimator with IS based on the decoupling approach \citep{is_mvsde} when used to estimate rare-event probabilities associated with a stochastic particle system in the mean-field limit. For this, we used stochastic optimal control theory to derive a zero-variance IS change of measure for the decoupled MV-SDE \eqref{eqn:decoupled_mvsde}. Our numerical experiments demonstrated that the above obtained IS control substantially reduces the variance of the DLMC estimator. In the rare-event regime, where the standard DLMC approach fails, our approach yields accurate estimates of rare-event probabilities with reduced computational effort. Our novel DLMC estimator has a computational cost of $\order{\tol_{\mathrm{r}}^{-4}}$ which is exactly the same complexity as that of the MC estimator proposed by \citep{mlmc_mvsde} for smooth, nonrare observables, while substantially reducing the associated constant for rare-event probabilities. Future works will involve extending the proposed IS scheme to the multidimensional case by using model-reduction techniques or stochastic gradient-based learning methods, leading to a more generalized and efficient algorithm. The presence of multiple discretization parameters in the decoupled MV-SDE hints toward the use of multilevel and multi-index MC methods coupled with IS to further reduce the work complexity of the DLMC estimator. 

\appendix
\label{sec:appendix}

\section{Proof of Lemma \ref{lemma:dyn_prog}}

\label{appendix:a1}

Assume that the minimizer for $C_{t,x}$ can be attained. Then, we prove the equality in \eqref{eqn:sde_dyn_prog} by proving both inequalities $\geq$ and $\leq$. First, consider $\geq$:

Let $\zeta^* \in \mathcal{Z}$ be the optimal control for $s \in (t,T)$ that minimizes the second moment. From the value function definition \eqref{sde_value_fxn}, we obtain
\begin{align}
    u(t,x) &= \mathbb{E}\Bigg[ \Bigg. G^2(Y_{\zeta^*}(T)) \exp \Bigg\{ \Bigg. -\int_t^T \norm{\zeta^*(s,Y_{\zeta^*}(s))}^2 \dd s \nonumber \\
    &\qquad - 2 \int_t^T \langle{\zeta^*}(s,Y_{\zeta^*}(s)),\dd W(s)\rangle \Bigg. \Bigg\} \quad \Bigg| \quad Y_{\zeta^*}(t)=x \Bigg. \Bigg] \nonumber \\
    &= \mathbb{E}\Bigg[ \Bigg.G^2(Y_{\zeta^*}(T)) \exp{-\int_t^{t+\delta} \norm{\zeta^*(s,Y_{\zeta^*}(s))}^2 \dd s - 2 \int_t^{t+\delta} \langle{\zeta^*}(s,Y_{\zeta^*}(s)),\dd W(s)\rangle} \nonumber \\
    & \qquad\exp{-\int_{t+\delta}^T \norm{\zeta^*(s,Y_{\zeta^*}(s))}^2 \dd s - 2 \int_{t+\delta}^T \langle{\zeta^*}(s,Y_{\zeta^*}(s)),\dd W(s)\rangle} \quad \Bigg| \quad Y_{\zeta^*}(t)=x \Bigg. \Bigg] \nonumber \\
    &= \mathbb{E}\Bigg[ \Bigg. \mathbb{E}\Bigg[ \Bigg. G^2(Y_{\zeta^*}(T)) \exp{-\int_t^{t+\delta} \norm{\zeta^*(s,Y_{\zeta^*}(s))}^2 \dd s - 2 \int_t^{t+\delta} \langle{\zeta^*}(s,Y_{\zeta^*}(s)),\dd W(s)\rangle} \nonumber \\
    \label{eqn:dyn_prog_step3}
    & \qquad \exp{-\int_{t+\delta}^T \norm{\zeta^*(s,Y_{\zeta^*}(s))}^2 \dd s - 2 \int_{t+\delta}^T \langle{\zeta^*}(s,Y_{\zeta^*}(s)),\dd W(s)\rangle} \nonumber \\
    &\qquad \quad \Bigg| \quad Y_{\zeta^*}(t)=x,\mathcal{F}_{t+\delta} \Bigg. \Bigg] \quad \Bigg| \quad Y_{\zeta^*}(t)=x \Bigg. \Bigg], 
\end{align}
where $\mathcal{F}_{t+\delta}$ is the Wiener process filtration until time $t+\delta$. Considering Markovianity for process $Y_{\zeta^*}$,  \eqref{eqn:dyn_prog_step3} can be expressed as

\begin{align}
    u(t,x) &= \mathbb{E}\Bigg[ \Bigg. \mathbb{E}\Bigg[ \Bigg. G^2(Y_{\zeta^*}(T)) \exp \Bigg\{ \Bigg.-\int_{t+\delta}^T \norm{\zeta^*(s,Y_{\zeta^*}(s))}^2 \dd s \nonumber\\
    & \qquad - 2 \int_{t+\delta}^T \langle{\zeta^*}(s,Y_{\zeta^*}(s)),\dd W(s)\rangle \Bigg. \Bigg\} \quad \Bigg| \quad Y_{\zeta^*}(t+\delta) = x_{t+\delta} \Bigg. \Bigg] \nonumber \\ 
    & \qquad \exp{-\int_t^{t+\delta} \norm{\zeta^*(s,Y_{\zeta^*}(s))}^2 \dd s - 2 \int_t^{t+\delta} \langle{\zeta^*}(s,Y_{\zeta^*}(s)),\dd W(s)\rangle} \quad \Bigg| \quad Y_{\zeta^*}(t)=x \Bigg. \Bigg]  \nonumber \\
    \label{eqn:dyn_prog_step5}
    &= \mathbb{E}\Bigg[ \Bigg. \exp{-\int_t^{t+\delta} \norm{\zeta^*(s,Y_{\zeta^*}(s))}^2 \dd s - 2 \int_t^{t+\delta} \langle{\zeta^*}(s,Y_{\zeta^*}(s)),\dd W(s)\rangle} \nonumber \\ 
    &\qquad C_{t+\delta,x_{t+\delta}}(\zeta^*) \quad \Bigg| \quad Y_{\zeta^*}(t)=x \Bigg. \Bigg] \cdot
\end{align}

As $\zeta^*$ may not be an optimal control from time $t+\delta$ to $T$ given $Y_{\zeta^*}(t+\delta)$, based on the definition of the value function \eqref{sde_value_fxn}, we get
\begin{equation}
\label{eqn:dyn_prog_step6}
    C_{t+\delta,x_{t+\delta}}(\zeta^*) \geq u(t+\delta,Y_{\zeta^*}(t+\delta)) \cdot
\end{equation}

By substituting \eqref{eqn:dyn_prog_step6} into \eqref{eqn:dyn_prog_step5}, we get
 
\begin{align}
    u(t,x) &\geq \mathbb{E}\Bigg[ \Bigg. \exp{-\int_t^{t+\delta} \norm{\zeta^*(s,Y_{\zeta^*}(s))}^2 \dd s - 2 \int_t^{t+\delta} \langle{\zeta^*}(s,Y_{\zeta^*}(s)),\dd W(s)\rangle} \nonumber \\
    \label{eqn:dyn_prog_step7}
    & \qquad u(t+\delta,Y_{\zeta^*}(t+\delta)) \quad \Bigg| \quad Y_{\zeta^*}(t)=x \Bigg. \Bigg] ; 
\end{align}

Taking the minimum over all controls $\zeta^*$ in $[t,t+\delta]$,
we obtain 

\begin{align}
    u(t,x) \geq \min_{\zeta^* :[t,t+\delta] \rightarrow \mathbb{R}^d} & \mathbb{E}\Bigg[ \Bigg. \exp{-\int_t^{t+\delta} \norm{\zeta^*(s,Y_{\zeta^*}(s))}^2 \dd s - 2 \int_t^{t+\delta} \langle{\zeta^*}(s,Y_{\zeta^*}(s)),\dd W(s)\rangle} \nonumber \\
    \label{eqn:dyn_prog_geq}
    & \qquad u(t+\delta,Y_{\zeta^*}(t+\delta)) \quad \Bigg| \quad Y_{\zeta^*}(t)=x \Bigg. \Bigg] \cdot
\end{align}

Next, let us consider the second inequality $\leq$. Let $\zeta^+$ be some arbitrary control from $t$ to $t+\delta$. Given $Y_{\zeta^*}(t+\delta)$, let $\zeta^*$ be the optimal control from $t+\delta$ to $T$, and define a new control $\zeta' = (\zeta^+,\zeta^*)$ over $[t,T]$. According to  \eqref{sde_value_fxn}, 
\begin{align}
    u(t,x) & \leq C_{t,x}(\zeta') \nonumber \\
    & \leq \mathbb{E}\Bigg[\Bigg.G^2(Y_{\zeta'}(T)) \exp{-\int_t^T \norm{\zeta'(s,Y_{\zeta'}(s))}^2 \dd s - 2 \int_t^T \langle{\zeta'}(s,Y_{\zeta'}(s)),\dd W(s)\rangle} \nonumber\\
    &\qquad \Bigg| \quad Y_{\zeta'}(t)=x\Bigg. \Bigg] \nonumber \\
    & \leq \mathbb{E}\Bigg[ \Bigg. G^2(Y_{\zeta'}(T)) \exp{-\int_t^{t+\delta} \norm{\zeta^+(s,Y_{\zeta^+}(s))}^2 \dd s - 2 \int_t^{t+\delta} \langle{\zeta^+}(s,Y_{\zeta^+}(s)),\dd W(s)\rangle}  \nonumber\\
    &\qquad \exp{-\int_{t+\delta}^T \norm{\zeta^*(s,Y_{\zeta^*}(s))}^2 \dd s - 2 \int_{t+\delta}^T \langle{\zeta^*}(s,Y_{\zeta^*}(s)),\dd W(s)\rangle} \quad \Bigg| \quad Y_{\zeta'}(t)=x \Bigg. \Bigg] \nonumber \\
    \label{eqn:dyn_prog_step10}
    & \leq \mathbb{E}\Bigg[ \Bigg. \mathbb{E}\Bigg[\Bigg. G^2(Y_{\zeta'}(T)) \exp{-\int_t^{t+\delta} \norm{\zeta^+(s,Y_{\zeta^+}(s))}^2 \dd s - 2 \int_t^{t+\delta} \langle{\zeta^+}(s,Y_{\zeta^+}(s)),\dd W(s)\rangle} \nonumber\\
    &\quad \exp{-\int_{t+\delta}^T \norm{\zeta^*(s,Y_{\zeta^*}(s))}^2 \dd s - 2 \int_{t+\delta}^T \langle{\zeta^*}(s,Y_{\zeta^*}(s)),\dd W(s)\rangle} \nonumber\\
    &\qquad \Bigg| \quad Y_{\zeta^*}(t)=x,\mathcal{F}_{t+\delta}\Bigg. \Bigg] \quad \Bigg| \quad Y_{\zeta'}(t)=x \Bigg. \Bigg] ,
\end{align}

and we can express \eqref{eqn:dyn_prog_step10} as

\begin{align}
    u(t,x) &\leq \mathbb{E}\Bigg[\Bigg. \mathbb{E}\Bigg[\Bigg.G^2(Y_{\zeta^*}(T)) \exp\Bigg\{\Bigg.-\int_{t+\delta}^T \norm{\zeta^*(s,Y_{\zeta^*}(s))}^2 \dd s \nonumber\\
    \label{eqn:dyn_prog_step11}
    &\qquad - 2 \int_{t+\delta}^T \langle{\zeta^*}(s,Y_{\zeta^*}(s)),\dd W(s)\rangle\Bigg.\Bigg\} \quad \Bigg| \quad Y_{\zeta^*}(t+\delta)\Bigg.\Bigg] \\ 
    &\qquad \exp{-\int_t^{t+\delta} \norm{\zeta^+(s,Y_{\zeta^+}(s))}^2 \dd s - 2 \int_t^{t+\delta} \langle{\zeta^+}(s,Y_{\zeta^+}(s)),\dd W(s)\rangle} \quad \Bigg| \quad Y_{\zeta'}(t)=x \Bigg.\Bigg] \nonumber \cdot
\end{align}

Considering optimality of control $\zeta^*$ in $[t+\delta,T]$,  we can express \eqref{eqn:dyn_prog_step11} as
\begin{align}
    u(t,x) \leq & \mathbb{E}\Bigg[ \Bigg. \exp{-\int_t^{t+\delta} \norm{\zeta^+(s,Y_{\zeta^+}(s))}^2 \dd s - 2 \int_t^{t+\delta} \langle{\zeta^+}(s,Y_{\zeta^+}(s)),\dd W(s)\rangle} \nonumber \\
    & \qquad u(t+\delta,Y_{\zeta^*}(t+\delta)) \quad \Bigg| \quad Y_{\zeta'}(t)=x \Bigg. \Bigg] \nonumber \cdot
\end{align}

By taking the minimum over all controls $\zeta^+ \in \mathcal{Z}$ over $[t,t+\delta]$, we get
\begin{align}
    u(t,x) \leq \min_{\zeta^+ :[t,t+\delta] \rightarrow \mathbb{R}^d} & \mathbb{E}\Bigg[ \Bigg. \exp{-\int_t^{t+\delta} \norm{\zeta^+(s,Y_{\zeta^+}(s))}^2 \dd s - 2 \int_t^{t+\delta} \langle{\zeta^+}(s,Y_{\zeta^+}(s)),\dd W(s)\rangle} \nonumber \\
    \label{eqn:dyn_prog_leq}
    & \qquad u(t+\delta,Y_{\zeta^*}(t+\delta)) \quad \Bigg| \quad Y_{\zeta'}(t)=x \Bigg. \Bigg] \cdot
\end{align}

Equations \eqref{eqn:dyn_prog_geq} and \eqref{eqn:dyn_prog_leq} prove the equality \eqref{eqn:sde_dyn_prog}. This completes the proof. 

\section{Proof of Theorem \ref{th:sde_optimal_control}}
\label{appendix:a2}

From Lemma \ref{lemma:dyn_prog}, the value function defined in \eqref{sde_value_fxn} satisfies the following equation: 
\begin{align}
    u(t,x) &= \min_{\zeta \in \mathcal{Z}} \mathbb{E}\Bigg[ \Bigg. \exp{-\int_t^{t+\delta} \norm{\zeta(s,Y_\zeta(s))}^2 \dd s} \exp{-2 \int_t^{t+\delta} \langle \zeta(s,Y_\zeta(s)),\dd W(s)\rangle} \nonumber \\ 
    \label{eqn:dyn_prog_re}
    & \qquad u(t+\delta,Y_\zeta(t+\delta)) \quad \Bigg| \quad Y_\zeta(t) = x  \Bigg. \Bigg] \cdot
\end{align}

We then use the Taylor series expansion of function $\exp{x}$ for small $\delta$,
\begin{align}
\label{eqn:sde_hjb_step1}
    \exp{-\int_t^{t+\delta} \norm{\zeta(s,Y_\zeta(s))}^2 \dd s} &= 1 -  \int_t^{t+\delta} \norm{\zeta(s,Y_\zeta(s))}^2 \dd s \\
    &\qquad + \sum_{m=2}^\infty \frac{(-1)^m}{m!} \left( \int_t^{t+\delta} \norm{\zeta(s,Y_\zeta(s))}^2 \dd s \right)^m \cdot \nonumber \\
    \label{eqn:sde_hjb_step2}
    \exp{-2 \int_t^{t+\delta} \langle \zeta(s,Y_\zeta(s)),\dd W(s)\rangle} &= 1 - 2 \int_t^{t+\delta} \langle \zeta(s,Y_\zeta(s)),\dd W(s)\rangle  \\
    + 2 \left( \int_t^{t+\delta} \langle \zeta(s,Y_\zeta(s)),\dd W(s)\rangle \right)^2 &+ \sum_{n=3}^\infty \frac{(-1)^n}{n!} \left( 2 \int_t^{t+\delta} \langle \zeta(s,Y_\zeta(s)),\dd W(s)\rangle \right)^n \cdot \nonumber  
\end{align}

Next, we write down It\^o's formula for $u(t+\delta,Y_\zeta(t+\delta))$,
\begin{align}
\label{eqn:sde_hjb_step3}
    u(t+\delta,Y_\zeta(t+\delta)) &= u(t,Y_\zeta(t)) + \int_t^{t+\delta} \left(\partial_t u + \langle b+\sigma \zeta, \nabla u \rangle + \frac{1}{2} (\sigma \sigma^T):\nabla^2 u \right) \dd t \nonumber \\
    &\qquad + \int_t^{t+\delta} \langle \sigma \nabla u, \dd W(s) \rangle \cdot
\end{align}

Substituting  \eqref{eqn:sde_hjb_step1}, \eqref{eqn:sde_hjb_step2}, and \eqref{eqn:sde_hjb_step3} in \eqref{eqn:dyn_prog_re}, we get
\begin{align}
    u(t,x) = \min_{\zeta \in \mathcal{Z}} \Bigg\{ & \Bigg. u(t,x) + u(t,x) \E{\int_t^{t+\delta} \norm{\zeta(s,Y_\zeta(s))}^2 \dd t \quad \Bigg| \quad Y_{\zeta}(t)=x} \nonumber\\ 
    &+ \E{\int_t^{t+\delta} \left(\partial_t u + \langle b+\sigma \zeta, \nabla u \rangle + \frac{1}{2} (\sigma \sigma^T):\nabla^2 u \right) \dd t \quad \Bigg| \quad Y_{\zeta}(t)=x} \nonumber\\
    \label{eqn:dmvsde_hjb_step4}
    &- 2 \E{\int_t^{t+\delta} \langle \sigma \zeta, \nabla u \rangle \dd t \quad \Bigg| \quad Y_{\zeta}(t)=x} + \mathcal{R} \Bigg. \Bigg\} \cdot
\end{align}

Here $\mathcal{R}$ is the residual term and one can see that

\begin{align*}
	\mathcal{R} &= 2u(t,x) \E{ \left( \int_t^{t+\delta} \norm{\zeta(s,Y_\zeta(s))}^2 \dd s \right) \left( \int_t^{t+\delta} \langle \zeta(s,Y_\zeta(s)),\dd W(s)\rangle \right) \quad \Bigg| \quad Y_{\zeta}(t)=x} \\
	&\quad - 2 \E{ \left( \int_t^{t+\delta} \left(\partial_t u + \langle b+\sigma \zeta, \nabla u \rangle + \frac{1}{2} (\sigma \sigma^T):\nabla^2 u \right) \dd t \right) \left( \int_t^{t+\delta} \langle \zeta(s,Y_\zeta(s)),\dd W(s)\rangle \right) \quad \Bigg| \quad Y_{\zeta}(t)=x} \\
	&\quad - \E{ \left( \int_t^{t+\delta} \norm{\zeta(s,Y_\zeta(s))}^2 \dd s \right) \left( \int_t^{t+\delta} \langle \sigma \nabla u, \dd W(s) \rangle \right) \quad \Bigg| \quad Y_{\zeta}(t)=x} \\
	&\quad - \frac{4}{3} u(t,x) \E{ \left(\int_t^{t+\delta} \langle \zeta(s,Y_\zeta(s)),\dd W(s)\rangle \right)^3 \quad \Bigg| \quad Y_{\zeta}(t)=x} + \mathrm{h.o.t} 
\end{align*}

One can easily see that $\mathcal{R} = \order{\delta^{\frac{3}{2}}}$ because

\begin{align*}
	&\E{ \left( \int_t^{t+\delta} \norm{\zeta(s,Y_\zeta(s))}^2 \dd s \right) \left( \int_t^{t+\delta} \langle \zeta(s,Y_\zeta(s)),\dd W(s)\rangle \right) \quad \Bigg| \quad Y_{\zeta}(t)=x} \\
	&\overset{\text{Cauchy-Schwarz}}{\leq} \E{\left( \int_t^{t+\delta} \norm{\zeta(s,Y_\zeta(s))}^2 \dd s \right)^2 \quad \Bigg| \quad Y_{\zeta}(t)=x}^{\frac{1}{2}} \E{\left( \int_t^{t+\delta} \langle \zeta(s,Y_\zeta(s)),\dd W(s)\rangle \right)^2 \quad \Bigg| \quad Y_{\zeta}(t)=x}^{\frac{1}{2}} \\
	&\overset{\text{Itô Isometry}}{=} \E{\left( \int_t^{t+\delta} \norm{\zeta(s,Y_\zeta(s))}^2 \dd s \right)^2 \quad \Bigg| \quad Y_{\zeta}(t)=x}^{\frac{1}{2}} \E{\int_t^{t+\delta} \norm{\zeta(s,Y_\zeta(s))}^2 \dd s \quad \Bigg| \quad Y_{\zeta}(t)=x}^{\frac{1}{2}} = \order{\delta^{\frac{3}{2}}}
\end{align*} 

and

\begin{align*}
	&\E{ \left(\int_t^{t+\delta} \langle \zeta(s,Y_\zeta(s)),\dd W(s)\rangle \right)^3 \quad \Bigg| \quad Y_{\zeta}(t)=x} \\
	&\quad \overset{\text{Burkholder-Davis-Gundy inequality}}{\leq} C_{\mathrm{BDG}} \E{\left( \int_t^{t+\delta} \norm{\zeta(s,Y_\zeta(s))}^2 \dd s \right)^{\frac{3}{2}} \quad \Bigg| \quad Y_{\zeta}(t)=x} = \order{\delta^{\frac{3}{2}}}
\end{align*}

One can use the same trick to bound even higher order terms in $\mathcal{R}$. From the regularity assumptions stated in Theorem \ref{th:sde_optimal_control}, we have, as $\delta \rightarrow 0$,

\begin{align}
    0 &= \min_{\zeta \in \mathcal{Z}} \Bigg\{ \norm{\zeta}^2 u + \partial_t u +  \langle b, \nabla u \rangle + \frac{1}{2} \left(\sigma\sigma^T\right) : \nabla^2 u - \langle \sigma \zeta , \nabla u \rangle \Bigg\} \nonumber \\
    \label{eqn:sde_hjb_step5}
    &= \partial_t u +  \langle b, \nabla u \rangle + \frac{1}{2} \left(\sigma\sigma^T\right) : \nabla^2 u + \min_{\zeta \in \mathcal{Z}} \Bigg\{ \norm{\zeta}^2 u - \langle \sigma \zeta , \nabla u \rangle \Bigg\} \cdot
\end{align}

Neglecting the trivial solution $u(t,x) = 0$, we obtain the minimizer for \eqref{eqn:sde_hjb_step5} as in \eqref{eqn:sde:hjb_minimizer}. By substituting optimal control $\zeta^*$ in \eqref{eqn:sde_hjb_step5}, we get \eqref{eqn:sde_hjb_form1}, which solves for value function $u$.

\section{Proof of Theorem \ref{th:dlmc_optimal_work}}
\label{appendix:a3}

We use the Lagrangian multiplier method to solve \eqref{eqn:dlmc_min_problem}, with corresponding Lagrangian, 
\begin{align}
    \mathcal{L} &= M_1 N_1 P^2 + M_1 M_2 N_2 P + \lambda_1 \left(\frac{C_p}{P} + \frac{C_{n_1}}{N_1} + \frac{C_{n_2}}{N_2} - \theta\tol_\mathrm{r} \abs{\E{G(X(T))}}\right) \nonumber \\
    \label{eqn:dlmc_lagrangian}
    &+ \lambda_2 \left(C_{\alpha}^2\left(\frac{C_1}{P M_1} + \frac{C_2}{M_1 M_2}\right) - (1-\theta)^2 \tol_\mathrm{r}^2 \abs{\E{G(X(T))}}^2 \right) , 
\end{align}
where $\lambda_1$ and $\lambda_2 \in \mathbb{R}$ are Lagrangian multipliers. Hence, we obtain optimality conditions for \eqref{eqn:dlmc_lagrangian} as follows:
\begin{align*}
    &\frac{\partial \mathcal{L}}{\partial M_1} = 0 \Longrightarrow N_1 P^2 + M_2 N_2 P = \lambda_2 C_\alpha^2 \left(\frac{C_1}{P M_1^2}+\frac{C_2}{M_2 M_1^2} \right), \\
    &\frac{\partial \mathcal{L}}{\partial M_2} = 0 \Longrightarrow M_1 N_2 P = \frac{\lambda_2 C_2 C_\alpha^2}{M_1 M_2^2}, \\
    &\frac{\partial \mathcal{L}}{\partial N_1} = 0 \Longrightarrow \frac{\lambda_1 C_{n_1}}{N_1^2} = M_1 P^2 ,\\
    &\frac{\partial \mathcal{L}}{\partial N_2} = 0 \Longrightarrow \frac{\lambda_1 C_{n_2}}{N_2^2} = M_1 M_2 P, \\
    & \frac{\partial \mathcal{L}}{\partial P} = 0 \Longrightarrow 2 M_1 N_1 P + M_1 M_2 N_2 = \frac{\lambda_1 C_p}{P^2} + \frac{\lambda_2 C_\alpha^2 C_1}{M_1 P^2}, \\
    & \frac{\partial \mathcal{L}}{\partial \lambda_1} = 0 \Longrightarrow \frac{C_p}{P} + \frac{C_{n_1}}{N_1} + \frac{C_{n_2}}{N_2} = \theta\tol_\mathrm{r} \abs{\E{G(X(T))}} ,\\
    & \frac{\partial \mathcal{L}}{\partial \lambda_2} = 0 \Longrightarrow C_{\alpha}^2\left(\frac{C_1}{P M_1} + \frac{C_2}{M_1 M_2}\right) = (1-\theta)^2 \tol_\mathrm{r}^2 \abs{\E{G(X(T))}}^2 \cdot
\end{align*}

By solving the above equations for $P,N_1,N_2,M_1,M_2,\lambda_1$, and $\lambda_2$, we get

\begin{align}
    P &= \frac{\left(C_p + \frac{\beta C_{n_1}}{\alpha} + \beta C_{n_2}\right)}{\theta\tol_\mathrm{r} \abs{\E{G(X(T))}}} , \nonumber \\
    N_1 &= \frac{\left(\frac{\alpha C_p}{\beta}+C_{n_1}+\alpha C_{n_2}\right)}{\theta\tol_\mathrm{r} \abs{\E{G(X(T))}}} , \nonumber \\
    \label{eqn:dlmc_optimal_params}
    N_2 &= \frac{\left(\frac{C_p}{\beta}+\frac{C_{n_1}}{\alpha} + C_{n_2}\right)}{\theta\tol_\mathrm{r} \abs{\E{G(X(T))}}}  ,\\
    M_1 &= \frac{\theta}{(1-\theta)^2} \frac{C_\alpha^2\left(C_1 + \frac{C_2}{\gamma}\right)}{\left(C_p + \frac{\beta C_{n_1}}{\alpha} + \beta C_{n_2}\right)\tol_\mathrm{r} \abs{\E{G(X(T))}}} , \nonumber\\
    M_2 &= \frac{\left(\gamma C_p + \frac{\gamma \beta C_{n_1} }{\alpha} + \beta \gamma C_{n_2}\right)}{\theta\tol_\mathrm{r} \abs{\E{G(X(T))}}} \nonumber ,
\end{align}
where constants 
\begin{align*}
    \alpha = \left(\frac{C_2}{C_1}\right)^{\frac{1}{3}} \left(\frac{C_{n_1}}{C_{n_2}}\right)^{\frac{2}{3}} , \quad
    \gamma = \left(\frac{C_2}{C_1}\right)^{\frac{2}{3}} \left(\frac{C_{n_1}}{C_{n_2}}\right)^{\frac{1}{3}} , \quad
    \beta = \frac{\left(\alpha^2 \frac{C_p}{C_{n_1}}\right)}{\left(\alpha + \gamma \right)} \cdot
\end{align*}

By substituting optimal parameters \eqref{eqn:dlmc_optimal_params} into \eqref{eqn:dlmc_cost_v2}, we get
\begin{equation}
    \mathcal{W} = M_1 N_1 P^2 + M_1 M_2 N_2 P = \order{\tol_\mathrm{r}^{-4}} \cdot
\end{equation}

\section{Estimating $V_{1,L}$ and $V_{2,L}$ for Adaptive DLMC}
\label{app:est_constants}

\begin{algorithm}[H]
    \caption{Estimating constants $V_{1,L}$ and $V_{2,L}$ for adaptive DLMC algorithm}
\label{alg:adlmc_est_constants}
    \SetAlgoLined
    \textbf{Inputs: } $P_L,N_L,M_1,M_2,\zeta(\cdot,\cdot)$; \\
    \For{$m_1=1,\ldots,M_1$}{
    Generate realization of random variables $\omega_{1:P_L}^{(m_1)}$; \\
    Generate $\mu^{P_L|N_L} \left( \omega_{1:P_L}^{(m_1)} \right)$ realization with $P_L$-particle system and $N_L$ time steps using \eqref{eqn:dmvsde_discrete_law}; \\
    \For{$m_2=1,\ldots,M_2$}{
    Generate realization of random variables $\tilde{\omega}^{(m_2)}$; \\
    Given $\mu^{P_L|N_L} \left( \omega_{1:P_L}^{(m_1)} \right)$ and $\zeta(\cdot,\cdot)$, solve decoupled MV-SDE with $N_L$ time steps using \eqref{eqn:dmvsde_euler_is_form2};\\
    Compute $G\left(\Bar{X}_\zeta^{P_L|N_L|N_L}(T)\right) \left( \omega_{1:P_L}^{(m_1)}, \tilde{\omega}^{(m_2)} \right)$; \\
    Compute $\mathbb{L}^{P_L|N_L|N_L} \left( \omega_{1:P_L}^{(m_1)}, \tilde{\omega}^{(m_2)} \right)$ using \eqref{eqn:dlmc_llhood_factor};
    }
    Approximate $\E{G\left(\Bar{X}_\zeta^{P_L|N_L|N_L}(T)\right) \mathbb{L}^{P_L|N_L|N_L} \mid \mu^{P_L|N_L} \left( \omega_{1:P_L}^{(m_1)} \right) }$ by $\frac{1}{M_2} \sum_{m_2=1}^{M_2} G\left(\Bar{X}_\zeta^{P_L|N_L|N_L}(T)\right) \mathbb{L}^{P_L|N_L|N_L} \left( \omega_{1:P_L}^{(m_1)}, \tilde{\omega}^{(m_2)} \right)$; \\
    Approximate $\Var{G\left(\Bar{X}_\zeta^{P_L|N_L|N_L}(T)\right) \mathbb{L}^{P_L|N_L|N_L} \mid \mu^{P_L|N_L} \left( \omega_{1:P_L}^{(m_1)} \right)}$ by sample variance of $\left\{G\left(\Bar{X}_\zeta^{P_L|N_L|N_L}(T)\right) \mathbb{L}^{P_L|N_L|N_L} \left( \omega_{1:P_L}^{(m_1)}, \tilde{\omega}^{(m_2)} \right)\right\}_{m_2=1}^{M_2}$; \\
    }
    Approximate $V_{1,L}$ by sample variance of  $\left\{\E{G\left(\Bar{X}_\zeta^{P_L|N_L|N_L}(T)\right) \mathbb{L}^{P_L|N_L|N_L} \mid \mu^{P_L|N_L} \left( \omega_{1:P_L}^{(m_1)} \right)} \right\}_{m_1=1}^{M_1}$ ;\\
    Approximate $V_{2,L}$ by $\frac{1}{M_1} \sum_{m_1=1}^{M_1} \Var{G\left(\Bar{X}_\zeta^{P_L|N_L|N_L}(T)\right) \mathbb{L}^{P_L|N_L|N_L} \mid \mu^{P_L|N_L} \left( \omega_{1:P_L}^{(m_1)} \right)}$;
\end{algorithm}

\bibliography{references}
\bibliographystyle{plainnat}
\setcitestyle{authoryear,open={(},close={)}}
\end{document}